\documentclass[11pt]{article}
\usepackage{amsfonts}

\usepackage{mathrsfs,amssymb,amsmath,amsfonts}
\usepackage{amsthm}
\usepackage{enumerate}
\usepackage{color}
\usepackage{indentfirst}
\usepackage{graphicx}          
\usepackage[dvips]{epsfig}    
\usepackage{wrapfig}                             
\usepackage{epstopdf}
\usepackage{amsmath}
\usepackage{amsfonts}
\usepackage{graphicx}
\usepackage{amssymb}
\usepackage{epsfig}
\usepackage{color}
\usepackage{mathrsfs}
\usepackage{enumerate}
\usepackage{subfigure}

\usepackage{caption}

\newtheorem{theorem}{Theorem}[section]
\newtheorem{lemma}[theorem]{Lemma}
\newtheorem{definition}[theorem]{Definition}
\newtheorem{proposition}[theorem]{Proposition}
\newtheorem{corollary}[theorem]{Corollary}
\newtheorem{remark}[theorem]{Remark}

\theoremstyle{definition}

\newenvironment{keywords}{{\bf Key words: }}{}
\newenvironment{AMS}{{\bf AMS subject classification: }}{}

\newcommand{\esssup}{\mathop{\mathrm{esssup}}}

\begin{document}

\title{Indefinite Mean-Field Type Linear-Quadratic Stochastic Optimal Control Problems  \footnote{ Foundation: This work is supported by the National Natural Science Foundation of
China (11871310 and 11801317), 
the National Key R\&D Program of China (2018YFA0703900), the Research Grants Council of Hong Kong under grant (15255416 and
15213218), the Colleges and Universities Youth Innovation Technology Program of Shandong Province (2019KJI011), and the PolyU-SDU Joint Research Centre on Financial Mathematics. }}

\author{Na Li\footnote{School of Statistics, Shandong University of Finance and Economics, Jinan 250014, China; email: {\tt naibor@163.com}}, \quad
Xun Li\footnote{Department of Applied Mathematics, The Hong Kong Polytechnic University, Hong Kong, China; email: {\tt malixun@polyu.edu.hk}} \quad and \quad
Zhiyong Yu\footnote{Corresponding author; School of Mathematics, Shandong University, Jinan 250100, China; email: {\tt yuzhiyong@sdu.edu.cn}}}

\maketitle
%
%
%

\begin{abstract}
This paper focuses on indefinite stochastic mean-field
linear-quadratic (MF-LQ, for short) optimal control
problems, which allow the weighting matrices for state and control
in the cost functional to be indefinite. The solvability of
stochastic Hamiltonian system and Riccati equations is presented
under both positive definite case and indefinite case. The optimal
controls in open-loop form and closed-loop form are obtained,
respectively. Moreover, the dynamic mean-variance problem can be
solved within the framework of the indefinite MF-LQ problem. Other
two examples shed light on the theoretical results
established.
\end{abstract}


\section{Introduction}
Historically, researchers have made many contributions to
McKean-Vlasov type stochastic differential equation (SDE, for short)
(\cite{Ahmed-Ding-1995,Ahmed-2007,Borkar-Kumar-2010,Chan-1994,Crisan-Xiong-2010,Huang-Malhame-Caines-2006,Kotelenez-Kurtz-2010}), which can be regarded as a kind of
mean-field SDE (MF-SDE, for short).  In recent years, stochastic
mean-field optimal control  problems, mean-field differential
games and their applications have attracted researchers'
attention. Andersson and Djehiche \cite{Andersson-Djehiche-2011}, and Buckdahn et al. \cite{Buckdahn-Djehiche-Li-2011} studied the maximum principle for SDEs of mean-field type, respectively. Buckdahn {\it et al.} \cite{Buckdahn-Djehiche-Li-Peng-2009} considered the mean-field backward SDE (MF-BSDE), Bensoussan {\it et al.} \cite{Bensoussan-Yam-Zhang-2015} obtained the unique solvability of mean-field type forward-backward SDE (MF-FBSDE). Recently, Duncan and Tembine \cite{Duncan-Tembine-2018} applied a direct method to discuss an MF-LQ game. Barreiro-Gomez et al. \cite{Barreiro-Duncan-Tembine-2019} investigated an MF-LQ game of jump-diffusion process with regime switching. 
This paper focuses on MF-LQ stochastic optimal control
problems for the indefinite weighting case, which generalize the
work of mean-field type optimal control problems with positive definite weighting case.

For the positive definite case, MF-LQ problems have been studied
widely over the past decade. Yong \cite{Yong-2013} considered an MF-LQ problem with
deterministic coefficients over a finite time horizon, and presented the optimal feedback using
a system of Riccati equations. Recently, there are some related
works following up Yong \cite{Yong-2013}  (see
\cite{Huang-Li-Yong-2015,Li-Sun-Yong-2016,Yong-2017,Wei-Yong-Yu-2019,Sun-Wang-2019}). Different from deterministic LQ problem, in the cost functional, the cost weighting matrices for the state and the control are allowed to be indefinite.  We notice that in the stochastic LQ setting, the cost functional with indefinite cost weighting matrices may still be convex in control. It is precisely this feature that determines whether an optimal control exists. Indefinite stochastic LQ theory has been extensively developed and has lots of interesting applications. Chen et al. \cite{Chen-Li-Zhou-1998} studied a kind of indefinite LQ problem based on Riccati equation. Rami et al. \cite{Rami-Moore-Zhou} showed that the solvability of the generalized Riccati equation is sufficient and necessary condition for the well-posedness of the indefinite LQ problem. Subsequent research includes various cases, and refer to \cite{Kohlmann-Zhou-2000,Qian-Zhou-2013,Sun-2017}.

 One of the motivations for indefinite MF-LQ problems comes from the mean-variance portfolio selection problem. Markowitz initially proposed and solved the mean-variance problem in the single-period setting in his Novel-Prize winning work \cite{Markowitz-1952,Markowitz-1959}, which is an important foundation of the development of modern finance. After Markowitz's pioneering work, the mean-variance model was extended to multi-period/continuous-time portfolio selection. If one wants to solve the mean-variance portfolio selection, she faces to two-objective: One is to minimize the difference between the terminal wealth and its expected value; the other one is to maximize her expected terminal wealth. Since there are two criteria in one cost functional, this stochastic control problem is significantly different from the classic LQ problem. The main reason is due to the variance term
  \[
 \mbox{Var}(X(T))=\mathbb E\big[X(T)-\mathbb E [X(T)]\big]^2
 \]
 essentially, which involves the nonlinear term of $(\mathbb E[X(T)])^2$.
In general, for nonlinear utility function $U(\cdot)$, there exists an essential difference between $\mathbb E[U(X(T))]$ and $U(\mathbb E [X(T)])$, which leads to the fundamental difficulty to deal with the latter one by dynamic programming. Li and Zhou \cite{Li-Zhou-2000} embedded this problem into an auxiliary stochastic LQ problem, which actually is one of indefinite LQ problems. In this paper, we re-visit the continuous-time mean-variance problem using the theoretical results of indefinite MF-LQ problems in a direct way (see the example in Section \ref{example1}).

Besides the dynamic mean-variance portfolio selection problem, there are many phenomena in finance and engineering fields which involve indefinite weighting parameters in the integral term as well as the terminal term. Another motivation is inspired by multi-objective optimization 
problems involving mean value. These problems can be converted into a single-objective problem by putting weights on the different objectives, which essentially are the indefinite mean-field optimization problems. For example, in a moving high-speed train, the controller wants to improve the speed as high as possible. Except for speeding up the train, the controller also wants to improve the resistance to the stochastic disturbance, which means that the state $X(\cdot)$ of train can not deviate too much from the mean value $\mathbb E[X(\cdot)]$. Therefore, there is a tradeoff between two objectives: One is to maximize the total speed $\mathbb E\int_0^T |u(t)|^2dt$, the other one is to minimize the variance over interval $[0,T]$ measured by $\mathbb E\int_0^T|X(t)-\mathbb E[X(t)]|^2dt$. We convert this multi-objective optimization 
problem into a single-objective problem as: 
 \[
 J(u(\cdot))=\mathbb E\int_0^T\Big\{\alpha \big|X(t)-\mathbb E[X(t)]\big|^2-\beta |u(t)|^2\Big\}dt
  \]
with $\alpha,\beta>0$. When the system is linear, this problem is a special case of indefinite MF-LQ problem.

  In literatures about indefinite LQ problem, the standard matrix inverse is involved in the Riccati equation, requiring the related term to be nonsingular. 
However, sometimes, the theory of Riccati equation is abstract and difficult. For example, the global solvability of Riccati equation (in the indefinite case or/and in the stochastic case) is often not simple. For this reason, we want to find another element with flexible restrictions instead of Riccati equation. Based on Yong \cite{Yong-2013} and inspired by Yu \cite{Yu-2013} and Huang and Yu \cite{Huang-Yu-2014}, we generalize the results of positive definite MF-LQ problem to the indefinite case by introducing a {\emph{relaxed compensator}}, which can be regarded as a generalization of the solution of Riccati equation. The presence of the relaxed compensator guarantees the well-posedness of MF-LQ problem. The open-loop and closed-loop optimal controls are also obtained under indefinite case. There are three main contributions of this paper:

\begin{enumerate}[(i)]
	\item Comparing with the solvability of Riccati equations, the relaxed compensator is defined under more flexible conditions (Condition (RC) in Section \ref{sec4}), which is more general.  
\item Based on the linear transformation involving relaxed
compensator, we analyze the unique solvability of a kind of MF-FBSDEs, which
does not satisfy the monotonicity condition in \cite{Bensoussan-Yam-Zhang-2015}.
\item We obtain the existing of relaxed compensator, which is a  sufficient and necessary condition for the solvability of Riccati equations. 
\end{enumerate}

Recently,  Sun \cite{Sun-2017} studied the MF-LQ problem under a uniform convexity condition, and showed that the convergence of a family of uniformly convex cost functionals is equivalent to the open-loop solvability of the MF-LQ problem. Different from the method in  \cite{Sun-2017}, this paper focuses on how to find a relaxed compensator to extend the condition of cost functional from positive case to the indefinite case. 

The rest of this paper is organized as follows. We present some
preliminaries and formulate an MF-LQ problem in Section \ref{sec2}.
Section \ref{sec3} is devoted to studying the MF-LQ
problem under positive definite case. Section \ref{sec4} focuses on
the indefinite MF-LQ problem, and derives the open-loop optimal
control and the optimal feedback control. Section \ref{sec5}
illustrates some applications including the dynamic mean-variance
problem and other two examples.

\section{Problem formulation and preliminaries}\label{sec2}

We denote by $\mathbb R^n$ the $n$-dimensional
Euclidean space. Let $\mathbb R^{n\times m}$ be the set of all
$(n\times m)$ matrices. Let $\mathbb S^n \subset \mathbb R^{n\times
n}$ be the collection of all symmetric matrices. As usual, if a
matrix $A\in \mathbb S^n$ is positive semidefinite (resp. positive
definite; negative semidefinite; negative definite), we denote
$A\geq 0$ (resp. $>0$; $\leq 0$; $<0$). All the positive
semidefinite (resp. negative semidefinite) matrices are collected by
$\mathbb S^n_+$ (resp. $\mathbb S^n_-$). Let $(\Omega, \mathcal F,
\mathbb P, \mathbb F)$ be a complete filtered probability space on
which a one-dimensional standard Brownian motion $W(\cdot)$ is
defined with $\mathbb F\equiv\{\mathcal F_t\}_{t\geq0}$ being its
natural filtration augmented by all $\mathbb P$-null sets.
For simplicity, we will restrict ourselves to the case of one-dimensional standard Brownian motion.
Some extensions to the case with multi-dimensional standard Brownian motion will be similarly derived examples in Section \ref{sec5}.
Let $T>0$ be a finite time horizon. Let $\mathbb H = \mathbb R^n$, $\mathbb R^{n\times m}$, $\mathbb S^n$, $\mathbb S^n_+$, etc.
We introduce the following notation which will be used in the paper:
\begin{itemize}
\item $L^\infty(0,T;\mathbb
H)$ is the space of $\mathbb H$-valued continuous functions
$\varphi(\cdot)$ such that
$\esssup\limits_{t\in[0,T]}|\varphi(t)|<\infty$.
\item $C^1([0,T];\mathbb H)$ is the space of $\mathbb H$-valued functions $ \varphi(\cdot)$ such that $\dot \varphi(\cdot)$ is continuous.
\item $L^2_{\mathcal F_T}(\Omega;\mathbb H)$ is the space of $\mathbb{H}$-valued $\mathcal F_T$-measurable random variables $\xi$ such that
$\mathbb{E}[ |\xi|^2 ] <\infty$.
\item $L^2_{\mathbb F}(0,T;\mathbb H)$ is the space of $\mathbb{H}$-valued $\mathbb F$-progressively measurable processes
$\varphi(\cdot)$ such that $\mathbb E\int_0^T |\varphi(t)|^2 dt
<\infty$.
\item $L^2_{\mathbb F}(\Omega;C([0,T];\mathbb H))$ is the space of
$\mathbb H$-valued $\mathbb F$-progressively measurable processes
$\varphi(\cdot)$ such that for almost all $\omega\in \Omega$,
$r\mapsto \varphi(r,\omega)$ is continuous and $\mathbb
E\left[\sup\limits_{t\in [0,T]} |\varphi(t)|^2\right] <\infty$.
\end{itemize}

\medskip

 Let $ \mathscr U[0,T] \equiv L^2_{\mathbb
F}(0,T;\mathbb R^m) $ denote the set of admissible controls. For any
initial state $x\in \mathbb R^n$ and any admissible control
$u(\cdot)\in \mathscr U[0,T]$, we consider the following controlled
MF-SDE:
\begin{equation}\label{Sec2_Sys}
\left\{
\begin{aligned}
& dX(t) = \Big\{ A(t)X(t) +\widetilde A(t)\mathbb E[X(t)]+B(t)u(t)+\widetilde B(t)\mathbb E [u(t)] \Big\} \\
& \qquad\qquad +\Big\{ C(t)X(t) +\widetilde C(t)\mathbb E [X(t)]+D(t)u(t)+\widetilde D(t)\mathbb E [u(t)] \Big\} dW(t),\quad t\in [0,T],\\
& X(0) = x,
\end{aligned}
\right.
\end{equation}
where $A(\cdot)$, $\widetilde A(\cdot)$, $C(\cdot)$, $\widetilde
C(\cdot) \in L^\infty(0,T;\mathbb R^{n\times n})$ and $B(\cdot)$,
$\widetilde B(\cdot)$, $D(\cdot)$, $\widetilde D(\cdot) \in
L^\infty(0,T;\mathbb R^{n\times m})$. By Proposition
2.6 in Yong \cite{Yong-2013} (see also Proposition 2.1 in
\cite{Yong-2017} and Proposition 2.2 in \cite{Wei-Yong-Yu-2019} for
wider versions), the MF-SDE \eqref{Sec2_Sys} admits a unique
solution $X(\cdot) \equiv X(\cdot;x,u(\cdot)) \in L^2_{\mathbb
F}(\Omega; C([0,T];\mathbb R^n))$. $X(\cdot)$ is called an
admissible trajectory corresponding to $u(\cdot)$, and
$(X(\cdot),u(\cdot))$ is called an admissible pair.

\medskip

Now, we present a cost functional as follows:
\begin{equation}\label{Sec2_Cost}
\begin{aligned}
J\big(x;u(\cdot)\big) =\ & \mathbb E \bigg\{ \int_0^T \Big[
\big\langle Q(t)X(t),\ X(t) \big\rangle +\big\langle \widetilde
Q(t)\mathbb E [X(t)],\ \mathbb E [X(t)] \big\rangle\\
& +2\big\langle S(t)u(t),\ X(t) \big\rangle +2\big\langle \widetilde
S(t)\mathbb E [u(t)],\ \mathbb E [X(t)] \big\rangle\\
& +\big\langle R(t)u(t),\ u(t) \big\rangle +\big\langle \widetilde
R(t)\mathbb E [u(t)],\ \mathbb E [u(t)] \big\rangle \Big] dt\\
& +\big\langle GX(T),\ X(T) \big\rangle +\big\langle \widetilde G
\mathbb E [X(T)],\ \mathbb E [X(T)] \big\rangle \bigg\},
\end{aligned}
\end{equation}
where $Q(\cdot)$, $\widetilde Q(\cdot) \in L^\infty(0,T;\mathbb
S^n)$, $S(\cdot)$, $\widetilde S(\cdot) \in L^\infty(0,T;\mathbb
R^{n\times m})$, $R(\cdot)$, $\widetilde R(\cdot)\in
L^\infty(0,T;\mathbb S^m)$, and $G$, $\widetilde G \in \mathbb S^n$.
It is clear that, for given $x\in\mathbb R^n$ and any
$u(\cdot)\in \mathscr U[0,T]$, $J(x;u(\cdot))$ is well defined.

\medskip

\noindent{\bf Problem (MF-LQ).} We introduce a family of MF-LQ stochastic optimal control problems: find an admissible
control $u^*(\cdot) \in \mathscr U[0,T]$ such that
\begin{equation*}
J\big(x;u^*(\cdot)\big) =\inf_{u(\cdot)\in \mathscr{U}[0,T]} J\big(x;u(\cdot) \big).
\end{equation*}
Problem (MF-LQ) is called {\it well-posed} if the
 infimum of $J\big( x;u(\cdot)
\big)$ over the set of admissible controls is finite.
If Problem (MF-LQ) is well-posed and the infimum of the cost functional is achieved by an
admissible control $u^*(\cdot)$, then Problem (MF-LQ) is said to be
{\it solvable} and $u^*(\cdot)$ is called an {\it optimal control}.
$X^*(\cdot) \equiv X\big(\cdot;x, u^*(\cdot)\big)$ is called the
{\it optimal trajectory} corresponding to $u^*(\cdot)$, and
$(X^*(\cdot), u^*(\cdot))$ is called an {\it optimal pair}.

\medskip

For simplicity, we use the following notation in this paper:
\[
\left\{
\begin{aligned}
& \widehat A(\cdot) = A(\cdot) +\widetilde A(\cdot),\quad \widehat
B(\cdot) = B(\cdot) +\widetilde B(\cdot), \quad \widehat C(\cdot) =
C(\cdot) +\widetilde C(\cdot),\quad \widehat D(\cdot) = D(\cdot)
+\widetilde D(\cdot),\\
& \widehat Q(\cdot) = Q(\cdot) +\widetilde Q(\cdot), \quad \widehat
S(\cdot) = S(\cdot) +\widetilde S(\cdot), \quad~ \widehat R(\cdot) =
R(\cdot) +\widetilde R(\cdot), \quad \widehat G = G +\widetilde G.
\end{aligned}
\right.
\]
Similar to Yong \cite{Yong-2013}, we give another version of
\eqref{Sec2_Sys} and \eqref{Sec2_Cost}. In detail, by taking
expectation $\mathbb E [\cdot]$ on both sides of
\eqref{Sec2_Sys}, we have
\begin{equation}\label{Sec2_Ex}
\left\{
\begin{aligned}
& d \mathbb E [X(t)] = \Big\{ \widehat A(t) \mathbb E [X(t)]
+\widehat B(t) \mathbb
E[u(t)] \Big\} dt,\quad t\in [0,T],\\
& \mathbb E [X(0)] = x.
\end{aligned}
\right.
\end{equation}
Then, the difference between $X(\cdot)$ and $\mathbb E [X(\cdot)]$
satisfies
\begin{equation}\label{Sec2_x-Ex}
\left\{
\begin{aligned}
 d\big( X(t)-\mathbb E [X(t)] \big)& = \Big\{ A(t)
\big(X(t)-\mathbb E [X(t)]\big) +B(t)\big(u(t)-\mathbb
E[u(t)]\big) \Big\} dt\\
& \qquad +\Big\{ C(t)\big(X(t)-\mathbb E [X(t)]\big)
+\widehat C(t)\mathbb E [X(t)] \\
& \qquad +D\big( u(t)-\mathbb E [u(t)] \big) +\widehat
D(t) \mathbb E [u(t)] \Big\} dW(t),\quad t\in [0,T],\\
 X(0)-\mathbb E [X(0)] &=0.
\end{aligned}
\right.
\end{equation}
It is clear that
the system consisting of \eqref{Sec2_x-Ex} and \eqref{Sec2_Ex} is
equivalent to the equation \eqref{Sec2_Sys}. Also, cost functional
\eqref{Sec2_Cost} can be rewritten into the following form
\begin{equation}\label{Sec2_Cost_xEx}
\begin{aligned}
J\big(x;u(\cdot)\big) =\ & \mathbb E \bigg\{ \int_0^T \Big[
\big\langle  Q(t)(X(t)-\mathbb E[X(t)]),X(t)-\mathbb E[X(t)] \big\rangle +\big\langle \widehat{
Q}(t)\mathbb E [X(t)],\ \mathbb E [X(t)] \big\rangle\\
& ~~~~~~~~+2\big\langle S(t)( u(t)-\mathbb E[ u(t)]), X(t)-\mathbb E[ X(t)]\big\rangle +2\big\langle \widehat
S(t)\mathbb E [u(t)],\ \mathbb E [X(t)] \big\rangle\\
& ~~~~~~~~+\big\langle R(t)( u(t)-\mathbb E[ u(t)]), u(t)-\mathbb E[ u(t)]\big\rangle +\big\langle \widehat
R(t)\mathbb E [ u(t)],\ \mathbb E [u(t)] \big\rangle \Big] dt\\
& ~~~~~~~~+\big\langle G(X(T)-\mathbb E[X(T)],\ X(T)-\mathbb E[X(T)] \big\rangle +\big\langle \widehat G
\mathbb E [X(T)],\ \mathbb E [X(T)] \big\rangle \bigg\}.
\end{aligned}
\end{equation}

For convenience, we
introduce the following notation:

\[
\left\{
\begin{aligned}
& \mathbf Q(t) = \left( \begin{array}{ccc} Q(t) & O\\ O & \widehat
Q(t)
\end{array} \right), \quad \mathbf S(t) = \left( \begin{array}{ccc} S(t) & O\\ O &
\widehat S(t)
\end{array} \right),\\
& \mathbf R(t) = \left( \begin{array}{ccc} R(t) & O\\ O & \widehat
R(t)
\end{array} \right), \quad \mathbf G = \left( \begin{array}{ccc} G & O\\ O &
\widehat G
\end{array} \right),
\end{aligned}
\right.
\]
where $O$ denotes zero matrices with appropriate
dimensions.
\medskip

For an $\mathbb S^n$-valued process $f(\cdot)$, if
$f(t) \geq 0$ (resp. $>0$; $\leq 0$; $<0$) for almost everywhere
$t\in [0,T]$, then we denote $f(\cdot)\geq 0$ (resp. $>0$; $\leq 0$;
$<0$). Moreover, if there exists a constant $\delta>0$ such that
$f(\cdot) -\delta I_n \geq 0$ (resp. $f(\cdot)+\delta I_n \leq 0$),
then we denote $f(\cdot) \gg 0$ (resp. $f(\cdot)\ll 0$), where $I_n$
denotes the $(n\times n)$ identity matrix. Now, for a given
quadruple of $(\mathbf Q(\cdot), \mathbf S(\cdot), \mathbf R(\cdot),
\mathbf G)$, we introduce a positive definite (PD, for short)
condition:

\medskip

\noindent{\bf Condition (PD).}
$
\left( \begin{array}{ccc} \mathbf Q(\cdot) & \mathbf S(\cdot) \\ \mathbf
S(\cdot)^\top & \mathbf R(\cdot)
\end{array} \right) \geq 0, \quad \mathbf R(\cdot) \gg 0, \quad
\mathbf G \geq 0,\qquad t\in [0,T].
$

\medskip

\noindent Here and hereafter, we use the superscript $\top$ to
denote the transpose of a matrix (or a vector).

\begin{remark}\label{Sec2_Rem}
It is clear that, if $(\mathbf Q(\cdot), \mathbf S(\cdot), \mathbf
R(\cdot), \mathbf G)$ satisfies Condition (PD), then we have
$J(x;u(\cdot)) \geq 0$ for any $x\in \mathbb R^n$ and
any $u(\cdot)\in\mathscr{U}[0,T]$. Hence, Problem (MF-LQ) is
well-posed.
\end{remark}

\medskip
\section{Problem (MF-LQ) in Positive Definite Case}\label{sec3}
In this section, we study this problem under Condition (PD). Now we turn our
attention to the issue of the solvability of Problem (MF-LQ).
Firstly, we consider the solvability in the open-loop form. For
simplicity of notation, we introduce a couple of linear functions:
for any $t\in [0,T]$, any $\theta = (x,u,y,z)$ and $\tilde\theta =
(\tilde x,\tilde u,\tilde y,\tilde z) \in \mathbb
R^{n+m+n+n}$, we define
\begin{equation}\label{Sec2.O_gPsi}
\left\{
\begin{aligned}
g(t,\theta, \tilde \theta) =\ & Q(t)x +\widetilde Q(t)\tilde
x +S(t)u +\widetilde S(t) \tilde u \\
& + A(t)^\top y +\widetilde A(t)^\top \tilde y +C^\top(t) z
+\widetilde C(t)^\top \tilde
z,\\
\Psi(t,\theta,\tilde\theta) =\ & S(t)^\top x +\widetilde
S(t)^\top \tilde x + R(t)u +\widetilde R(t)\tilde u\\
& +B(t)^\top y +\widetilde B(t)^\top \tilde y +D(t)^\top z
+\widetilde D(t)^\top \tilde z.
\end{aligned}
\right.
\end{equation}

\begin{lemma}\label{Sec2.O_Lem_Necessity} \sl
Let $(X^*(\cdot), u^*(\cdot))$ be an
optimal pair of Problem (MF-LQ) with initial state $x\in\mathbb R^n$. Let
$(Y(\cdot),Z(\cdot)) \in L^2_{\mathbb F}(\Omega; C([0,T];\mathbb
R^n)) \times L^2_{\mathbb F}(0,T;\mathbb R^n)$ be the unique
solution to the following mean-field backward stochastic
differential equation (MF-BSDE, for short):
\begin{equation}\label{Sec2.O_MF-BSDE}
\left\{
\begin{aligned}
& dY(t) = -g\big(t, \Theta^*(t), \mathbb E [\Theta^*(t)]\big) dt +Z(t) dW(t),\quad t\in [0,T],\\
& Y(T) = GX^*(T) +\widetilde G\mathbb E [X^*(T)],
\end{aligned}
\right.
\end{equation}
where $\Theta^*(\cdot) = (X^*(\cdot),u^*(\cdot),Y(\cdot),Z(\cdot))$
and $\mathbb E [\Theta^*(\cdot)] = (\mathbb E [X^*(\cdot)],
\mathbb E [u^*(\cdot)], \mathbb E [Y(\cdot)], \mathbb
E[Z(\cdot)])$. Then the following stationarity condition holds:
\begin{equation}
\Psi \big(t, \Theta^*(t), \mathbb E [\Theta^*(t)]\big) =0,\quad
t\in [0,T].
\end{equation}
\end{lemma}

\begin{proof}
By Proposition 2.6 in \cite{Wei-Yong-Yu-2019}, MF-BSDE
\eqref{Sec2.O_MF-BSDE} admits a unique solution
\[
(Y(\cdot),Z(\cdot)) \in L^2_{\mathbb F}(\Omega; C([0,T];\mathbb
R^n)) \times L^2_{\mathbb F}(0,T;\mathbb R^n).
\]
Besides the optimal pair $(X^*(\cdot),u^*(\cdot))$, we consider
also another arbitrary admissible pair $(X(\cdot), u(\cdot))$.
Let
\[
\Delta X(\cdot) = X(\cdot) -X^*(\cdot),\qquad \Delta u(\cdot) =
u(\cdot) -u^*(\cdot).
\]
Then $\Delta X(\cdot)$ satisfies the following MF-SDE:
\[
\left\{
\begin{aligned}
& d\Delta X = \Big\{ A\Delta X +\widetilde A \mathbb E [\Delta X]+B\Delta u +\widetilde B\mathbb E [\Delta u]\Big\}dt \\
& \qquad\quad +\Big\{ C\Delta X +\widetilde C \mathbb E [\Delta X]+D\Delta u +\widetilde D\mathbb E [\Delta u] \Big\}dW(t), \quad t\in [0,T], \\
& \Delta X(0) =0,
\end{aligned}
\right.
\]
which is in the form of  MF-SDE \eqref{Sec2_Sys} with the initial state $\Delta X(0) =0$. By
applying It\^{o}'s formula to $\langle \Delta X(\cdot), Y(\cdot)
\rangle$ on the interval $[0,T]$ and taking expectation, we have
\[
\begin{aligned}
& \mathbb E  \bigg\{ \int_0^T \Big[ \big\langle QX^*,\ \Delta X
\big\rangle +\big\langle \widetilde Q\mathbb E [X^*],\ \mathbb
E[\Delta X] \big\rangle +\big\langle Su^*,\ \Delta X \big\rangle
+\big\langle \widetilde S\mathbb E [u^*],\ \mathbb
E[\Delta X] \big\rangle \Big] dt\\
& +\big\langle GX^*(T),\ \Delta X(T) \big\rangle +\big\langle
\widetilde G \mathbb E [X^*(T)],\ \mathbb E [\Delta
X(T)]\big\rangle \bigg\}\\
=\ & \mathbb E \int_0^T \big\langle \Delta u,\ B^\top Y +\widetilde
B^\top \mathbb E [Y] +D^\top Z +\widetilde D^\top\mathbb E [Z]
\big\rangle dt.
\end{aligned}
\]
Adding $\mathbb E  \int_0^T [ \langle S\Delta u,\ X^* \rangle
+\langle \widetilde S\mathbb E [\Delta u],\ \mathbb E [X^*]
\rangle + \langle Ru^*,\ \Delta u \rangle +\langle \widetilde
R\mathbb E [u^*],\ \mathbb E [\Delta u] \rangle] dt$ on both sides
of the above equation leads to
\[
\begin{aligned}
& \mathbb E  \bigg\{ \int_0^T \Big[ \big\langle QX^*,\ \Delta X
\big\rangle +\big\langle \widetilde Q\mathbb E [X^*],\ \mathbb
E[\Delta X] \big\rangle +\big\langle Su^*,\ \Delta X
\big\rangle +\big\langle S\Delta u,\ X^* \big\rangle\\
& \qquad +\big\langle \widetilde S\mathbb E [u^*],\ \mathbb
E[\Delta X] \big\rangle +\big\langle \widetilde S \mathbb
E[\Delta u], \mathbb E [X^*] \big\rangle +\big\langle Ru^*,\
\Delta u \big\rangle +\big\langle \widetilde R\mathbb E [u^*],\
\mathbb
E[\Delta u] \big\rangle \Big] dt\\
& \qquad +\big\langle GX^*(T),\ \Delta X(T) \big\rangle +\big\langle
\widetilde G \mathbb E [X^*(T)],\ \mathbb E [\Delta
X(T)]\big\rangle \bigg\} = \mathbb E \int_0^T \big\langle \Delta
u,\ \Psi \big(\Theta^*, \mathbb E [\Theta^*]\big) \big\rangle dt.
\end{aligned}
\]
We note that $\langle QX, \ X\rangle -\langle QX^*,\ X^* \rangle =
\langle Q\Delta X,\ \Delta X \rangle +2\langle QX^*,\ \Delta X
\rangle$, $\langle Su,\ X\rangle - \langle Su^*,\ X^* \rangle =
\langle S\Delta u,\ \Delta X \rangle +\langle Su^*,\ \Delta X
\rangle +\langle S\Delta u,\ X^* \rangle$ and so on. Using the above equation, we reduce the difference between
$J(x;u(\cdot))$ and $J(x;u^*(\cdot))$ to
\begin{equation}\label{Sec2.O_Lem_Eq1}
J\big(x;u(\cdot) \big) -J\big(x;u^*(\cdot) \big) =
J\big(x;\Delta u(\cdot) \big) +2\mathbb E \int_0^T \big\langle\Delta u,\ \Psi \big(t, \Theta^*(t), \mathbb E [\Theta^*(t)]\big)\big\rangle dt.
\end{equation}
Hence, for any $\alpha\in\mathbb{R}$ and any $u(\cdot) \in
\mathscr{U}[0,T]$, we have
\[
J\big(x;u^*(\cdot) +\alpha u(\cdot) \big) -J\big(x;
u^*(\cdot) \big)= \alpha^2 J\big(x;u(\cdot) \big) +2\alpha \mathbb
E \int_0^T \big\langle u, \  \Psi \big(t, \Theta^*(t), \mathbb
E[\Theta^*(t)]\big) \big\rangle dt.
\]
Since $u^*(\cdot)$ is optimal, the above equation implies
\[
\mathbb E \int_0^T \big\langle u(t),\ \Psi \big(t, \Theta^*(t),
\mathbb E [\Theta^*(t)]\big) \big\rangle dt =0,\quad \mbox{for all
} u(\cdot) \in \mathscr U[0,T],
\]
therefore $\Psi \big(\cdot, \Theta^*(\cdot), \mathbb
E[\Theta^*(\cdot)]\big) =0$. We complete the proof.
\end{proof}

Denote
\begin{equation*}
M^2_{\mathbb F}(0,T) = L^2_{\mathbb F}(\Omega;C([0,T];\mathbb R^n))
\times \mathscr U[0,T] \times L^2_{\mathbb F}(\Omega;C([0,T];\mathbb
R^n)) \times L^2_{\mathbb F}(0,T;\mathbb R^n).
\end{equation*}

\begin{theorem}\label{Sec2.O_THM_Solution} \sl
Assume that the quadruple $(\mathbf Q(\cdot), \mathbf S(\cdot),
\mathbf R(\cdot), \mathbf G)$ satisfies Condition (PD). Then, for a
given $x\in \mathbb R^n$, the following stochastic Hamiltonian
system
\begin{equation}\label{Sec2.O_Hamil_Sys}
\left\{
\begin{aligned}
& 0 = \Psi\big( \Theta^*, \mathbb E [\Theta^*] \big),\quad t\in [0,T],\\
& dX^* = \Big\{ AX^* +\widetilde A \mathbb E [X^*] +Bu^* + \widetilde B\mathbb E [u^*] \Big\}dt \\
& \qquad\quad +\Big\{ CX^* +\widetilde C\mathbb E [X^*] + Du^* + \widetilde D\mathbb E [u^*] \Big\}dW(t), \quad t\in [0,T],\\
& dY = -g\big(\Theta^*, \mathbb E [\Theta^*]\big) dt + ZdW(t), \quad t\in [0,T],\\
& X^*(0) = x,\quad Y(T) = GX^*(T) +\widetilde G\mathbb E[X^*(T)]
\end{aligned}
\right.
\end{equation}
admits a unique solution $\Theta^*(\cdot) \in M^2_{\mathbb F}(0,T)$.
Moreover, $(X^*(\cdot),u^*(\cdot))$ is the unique optimal pair of
Problem (MF-LQ).
\end{theorem}

\begin{proof}
Under Condition (PD), from Theorem 3.4 in \cite{Wei-Yong-Yu-2019},
the Hamiltonian system \eqref{Sec2.O_Hamil_Sys} admits a unique
solution $\Theta^*(\cdot) =
(X^*(\cdot),u^*(\cdot),Y(\cdot),Z(\cdot))$. Now, we prove that
$(X^*(\cdot), u^*(\cdot))$ is an optimal pair of Problem (MF-LQ).
For any another admissible pair $(X(\cdot), u(\cdot))$, we adopt the
notation and the derivation procedure of Lemma
\ref{Sec2.O_Lem_Necessity}. Precisely, we start from
\eqref{Sec2.O_Lem_Eq1}. It is clear that $J(x;\Delta
u(\cdot)) \geq 0$ and $\Psi (\cdot, \Theta^*(\cdot), \mathbb
E[\Theta^*(\cdot)]) =0$. Therefore,
\[
J\big(x;u(\cdot) \big) -J\big(x;u^*(\cdot)\big) \geq 0.
\]
Due to the arbitrariness of $u(\cdot)$, we prove the optimality of $u^*(\cdot)$.

Now, we turn to the uniqueness of the optimal control. Let $(\bar
X(\cdot),\bar u(\cdot)) \in L^2_{\mathbb F}(\Omega;C([0,T];\mathbb
R^n)) \times \mathscr U[0,T]$ be another optimal pair. By Lemma
\ref{Sec2.O_Lem_Necessity}, there exists a pair of processes $(\bar
Y(\cdot),\bar Z(\cdot)) \in L^2_{\mathbb F}(\Omega;C([0,T];\mathbb
R^n)) \times L^2_{\mathbb F}(0,T;\mathbb R^n)$ such that the
quadruple $\bar\Theta(\cdot) = (\bar X(\cdot), \bar u(\cdot), \bar
Y(\cdot), \bar Z(\cdot))$ also solves Hamiltonian system
\eqref{Sec2.O_Hamil_Sys}. By the uniqueness of
\eqref{Sec2.O_Hamil_Sys}, we obtain $(\bar X(\cdot),\bar u(\cdot)) =
(X^*(\cdot),u^*(\cdot))$. This implies the desired result.
\end{proof}

In the rest of this section, we derive the solvability of the
corresponding Riccati equations to construct a feedback form of the
optimal control $u^*(\cdot)$. For simplicity of notation, let us
define $\Gamma: [0,T] \times\mathbb{S}^n \rightarrow
\mathbb{R}^{m\times n}$ and $\widehat \Gamma: [0,T] \times
\mathbb{S}^n \times \mathbb{S}^n \rightarrow \mathbb{R}^{m\times n}$ by
\begin{equation}\label{Sec2.F_Gamma}
\left\{\begin{aligned}
& \Gamma\big(t, P\big) = - \big[D(t)^\top PD(t) +R(t)\big]^{-1}
\big[PB(t) +C(t)^\top
PD(t) +S(t)\big]^\top,\\
& \widehat\Gamma\big(t, P, \widehat P\big) = - \big[ \widehat D(t)^\top P
\widehat D(t) +\widehat R(t) \big]^{-1} \big[ \widehat P \widehat B(t)
+\widehat C(t)^\top P \widehat D(t) +\widehat S(t) \big]^\top.
\end{aligned}\right.
\end{equation}

\begin{theorem}\label{Sec2.F_THM_Solution} \sl
Assume that the quadruple $(\mathbf Q(\cdot), \mathbf S(\cdot),
\mathbf R(\cdot), \mathbf G)$ satisfies Condition (PD). Then the
following (decoupled) system of Riccati equations (with $t$
suppressed)
\begin{equation}\label{Sec2.F_Riccati_Eq1}
\left\{
\begin{aligned}
& \dot P +PA +A^\top P +C^\top PC +Q -\Gamma(P)^\top \big[D^\top PD +R\big] \Gamma(P) = 0,\quad t\in [0,T],\\
& P(T) = G,\\
& D^\top PD +R \gg 0, \qquad t\in [0,T]
\end{aligned}
\right.
\end{equation}
and
\begin{equation}\label{Sec2.F_Riccati_Eq2}
\left\{
\begin{aligned}
& \dot {\widehat P} +\widehat P \widehat A +\widehat A^\top \widehat P +\widehat C^\top P
\widehat C +\widehat Q -\widehat \Gamma(P, \widehat P)^\top \big[ \widehat
D^\top P \widehat D +\widehat R \big] \widehat\Gamma(P, \widehat P)
= 0,\quad t\in [0,T],\\
& \widehat P (T) = \widehat G, \\
& \widehat D^\top P\widehat D + \widehat R \gg 0,\qquad t\in [0,T]
\end{aligned}
\right.
\end{equation}
admits a unique pair of solutions $(P(\cdot), \widehat P(\cdot))$ taking
values in $\mathbb S^n_+ \times \mathbb S^n_+$. Moreover, for a given
$x\in\mathbb R^n$, the unique optimal  control $u^*(\cdot)$ of
Problem (MF-LQ) has the following feedback
form:
\begin{equation}\label{Sec2.F_Optimal_Control}
u^* = \Gamma\big(P\big) \big( X^* -\mathbb E [X^*] \big)
+\widehat\Gamma\big(P,\widehat P\big) \mathbb E [X^*],\quad t\in [0,T],
\end{equation}
where $X^*(\cdot)$ is determined by
\begin{equation}\label{Sec2.F_Optimal_State}
\left\{
\begin{aligned}
& dX^* = \Big\{ \big( A+B\Gamma(P) \big) \big( X^* -\mathbb E [X^*]\big) +\big( \widehat A +\widehat B\widehat\Gamma(P,\widehat P) \big)\mathbb E[X^*]\Big\} dt\\
& \qquad\quad +\Big\{ \big( C+D\Gamma(P) \big) \big( X^* -\mathbb E[X^*] \big) +\big( \widehat C +\widehat D\widehat\Gamma(P,\widehat P)\big)\mathbb E[X^*] \Big\}dW(t), \quad t\in [0,T],\\
& X^*(0) = x.
\end{aligned}
\right.
\end{equation}
Moreover,
\begin{equation*}
\inf_{u(\cdot)\in \mathscr U[0,T]}J(x;u^*(\cdot))=J\big(x; u(\cdot) \big)=\big\langle\widehat P(0)x,x\big\rangle.
\end{equation*}
\end{theorem}

\begin{proof}
If the quadruple $(\mathbf Q(\cdot), \mathbf S(\cdot), \mathbf
R(\cdot), \mathbf G)$ satisfies Condition (PD), the Riccati equation
\eqref{Sec2.F_Riccati_Eq1} is the standard case of Yong and Zhou
\cite{Yong-Zhou-1999}. Therefore, there exists a unique solution $P(\cdot)\in C^1([0,T];\mathbb S^n_+)$. Next, a short calculation for \eqref{Sec2.F_Riccati_Eq2} yields
\begin{equation}\label{Riccatirewrite}
\left\{
\begin{aligned}
    &~\dot {\widehat P}+\widehat P\Big[\widehat A-\widehat B\Big(\widehat R^{-1}\widehat S^\top-\big[ \widehat D^\top P
\widehat D +\widehat R \big]^{-1}\widehat D^\top P\big[\widehat C-\widehat D\widehat R^{-1}\widehat S^\top\big]\Big)\Big]\\
&~~~~+\Big[\widehat A-\widehat B\Big(\widehat R^{-1}\widehat S^\top-\big[\widehat D^\top P
\widehat D +\widehat R \big]^{-1}\widehat D^\top P\big[\widehat C-\widehat D\widehat R^{-1}\widehat S^\top\big]\Big)\Big]^\top\widehat P\\
&~~~~+\big(\widehat C-\widehat D\widehat R^{-1}\widehat S^\top)^\top\big(P-P\widehat D\big[\widehat D^\top P
\widehat D +\widehat R \big]^{-1}\widehat D^\top P
\big)\big(\widehat C-\widehat D\widehat R^{-1}\widehat S^\top)\\
&~~~~+\widehat Q-\widehat S\widehat R^{-1}\widehat S^\top-\widehat P\widehat B\big[\widehat D^\top P
\widehat D +\widehat R \big]^{-1}\widehat B^\top \widehat P=0,\\
&~\widehat P(T)=\widehat G.
\end{aligned}
\right.
\end{equation}
Since $(\mathbf Q(\cdot), \mathbf S(\cdot), \mathbf R(\cdot), \mathbf G)$ satisfies Condition (PD), we have
\begin{equation*}
\left\{
    \begin{aligned}
        &~\big(\widehat C-\widehat D\widehat R^{-1}\widehat S^\top)^\top\big(P-P\widehat D\big[\widehat D^\top P
\widehat D +\widehat R \big]^{-1}\widehat D^\top P
\big)\big(\widehat C-\widehat D\widehat R^{-1}\widehat S^\top)+\widehat Q-\widehat S\widehat R^{-1}\widehat S^\top\geq0,\\
&~\widehat D^\top P
\widehat D +\widehat R \gg 0, \quad \widehat G\geq0,
    \end{aligned}
    \right.
\end{equation*}
Riccati equation \eqref{Riccatirewrite} admits a unique solution $\widehat P(\cdot)\in C^1([0,T];\mathbb S^n_+)$. Then, the system of Riccati equations \eqref{Sec2.F_Riccati_Eq1}-\eqref{Sec2.F_Riccati_Eq2} admits a unique solution $(P(\cdot),\widehat P(\cdot))\in \big(C^1([0,T];\mathbb S^n_+)\big)^2$.

Next, we will prove that $(X^*,u^*)$ is the optimal pair of Problem (MF-LQ). We split the cost functional \eqref{Sec2_Cost_xEx} into two parts:
\begin{equation}\label{Jsplit}
    J(x;u(\cdot))=J_1(x;u(\cdot))+J_2(x;u(\cdot)),
\end{equation}
with
\begin{equation*}\label{J1}
    \begin{aligned}
        J_1(x;u(\cdot))=\mathbb E\int_0^T\big[\langle QX_1,~X_1\rangle+2\langle Su_1,~X_1\rangle+\langle Ru_1,~u_1\rangle\big]dt+\mathbb E\langle GX_1(T),~X_1(T)\rangle
    \end{aligned}
\end{equation*}
and
\begin{equation*}
    \begin{aligned}
        J_2(x;u(\cdot))=\mathbb E\int_0^T\big[\langle \widehat QX_2,~X_2\rangle+2\langle \widehat Su_2,~X_2\rangle+\langle \widehat Ru_2,~u_2\rangle\big]dt+\mathbb E\langle \widehat GX_2(T),~X_2(T)\rangle,
    \end{aligned}
\end{equation*}
where $X_1(\cdot)=X(\cdot)-\mathbb E [X(\cdot)]$,
$X_2(\cdot)=\mathbb E [X(\cdot)]$, $u_1(\cdot)=u(\cdot)-\mathbb E
[u(\cdot)]$, and $u_2(\cdot)=\mathbb E [u(\cdot)]$.

Now, we deal with $J(x;u(\cdot))$ by two steps. 

\noindent {\bf Step 1:} Let $P(\cdot)$ be the solution of Riccati equation \eqref{Sec2.F_Riccati_Eq1}. Applying It\^o's formula to $\langle PX_1,~X_1\rangle$, we obtain
\begin{equation*}
    \begin{aligned}
        &~d\langle PX_1,~X_1\rangle\\&=\Big\{\big\langle(\dot P+PA+A^\top P+C^\top PC)X_1,~X_1\big\rangle+2\big\langle(PB+C^\top PD)u_1,~X_1\big\rangle+\big\langle D^\top PDu_1,~u_1\big\rangle\\
        &~~~~~~+\big\langle\widehat C^\top P\widehat CX_2,~X_2\big\rangle+2\big\langle\widehat C^\top P \widehat Du_2,~X_2\big\rangle+\big\langle\widehat D^\top P\widehat Du_2,~u_2\big\rangle\Big\}dt+\{...\}dW(t).
    \end{aligned}
\end{equation*}
Integrating on $[0,T]$ and   taking expectation $\mathbb E[\cdot]$ on both sides of the above equility,  we have

\begin{equation}\label{PX_1X_1}
    \begin{aligned}
        &~~~\mathbb E\big\langle P(T)X_1(T),~X_1(T)\big\rangle\\
        &=\mathbb E\int_0^T\Big\{\big\langle(\dot P+PA+A^\top P+C^\top PC)X_1,~X_1\big\rangle\\
        &~~~~~~+2\big\langle(PB+C^\top PD)u_1,~X_1\big\rangle+\big\langle D^\top PDu_1,~u_1\big\rangle\Big\}dt\\
        &~~~+\mathbb E\int_0^T\Big\{\big\langle\widehat C^\top P\widehat CX_2,~X_2\big\rangle+2\big\langle\widehat C^\top P \widehat Du_2,~X_2\big\rangle+\big\langle\widehat D^\top P\widehat Du_2,~u_2\big\rangle\Big\}dt.
            \end{aligned}
\end{equation}
Substituting \eqref{PX_1X_1} into $J_1(x;u(\cdot))$ yields
\begin{equation}\label{J1_2}
    \begin{aligned}
        J_1(x;u(\cdot))&=\mathbb E\int_0^T\Big\{\big\langle(\dot P+PA+A^\top P+C^\top PC+Q)X_1,~X_1\big\rangle\\
        &~~~~~~~+2\big\langle(PB+C^\top PD+S)u_1,~X_1\big\rangle+\big\langle (D^\top PD+R)u_1,~u_1\big\rangle\Big\}dt\\
        &~~~+\mathbb E\int_0^T\Big\{\big\langle\widehat C^\top P\widehat CX_2,~X_2\big\rangle+2\big\langle\widehat C^\top P \widehat Du_2,~X_2\big\rangle+\big\langle\widehat D^\top P\widehat
        Du_2,~u_2\big\rangle\Big\}dt    \end{aligned}
\end{equation}
Using the square completion method about $X_1$ and $u_1$ on \eqref{J1_2}, provided $D^\top PD+R>0$, we have
\begin{equation}\label{J1squared}
\begin{aligned}
    J_1(x;u(\cdot))&=\mathbb E\int_0^T\Big\{\big\langle (D^\top PD+R)[u_1-\Gamma(P)X_1],~[u_1-\Gamma(P)X_1]\big\rangle\\
    &~~~~~~+\big\langle\big[\dot P+PA+A^\top P+C^\top PC+Q-\Gamma(P)^\top(D^\top PD+R)\Gamma(P)\big]X_1,~X_1\big\rangle\Big\}dt\\
    &~~~~~~+\mathbb E\int_0^T\Big\{\big\langle\widehat C^\top P\widehat CX_2,~X_2\big\rangle+2\big\langle\widehat C^\top P \widehat Du_2,~X_2\big\rangle+\big\langle\widehat D^\top P\widehat Du_2,~u_2\big\rangle\Big\}dt\\
    &~~~ = \mathbb E\int_0^T\big\langle (D^\top PD+R)[u_1-\Gamma(P)X_1],~[u_1-\Gamma(P)X_1]\big\rangle dt\\
    &~~~~~~+\mathbb E\int_0^T\Big\{\big\langle\widehat C^\top P\widehat CX_2,~X_2\big\rangle+2\big\langle\widehat C^\top P \widehat Du_2,~X_2\big\rangle+\big\langle\widehat D^\top P\widehat Du_2,~u_2\big\rangle\Big\}dt.
\end{aligned}
\end{equation}
%
Substituting \eqref{J1squared} into \eqref{Jsplit} leads to
\[
\begin{aligned}
J(x;u(\cdot))&=J_1(x;u(\cdot))+J_2(x;u(\cdot))\\
    &=\mathbb E\int_0^T\Big\{\big\langle (D^\top PD+R)[u_1-\Gamma(P)X_1],~[u_1-\Gamma(P)X_1]\big\rangle\Big\}dt+\widehat
    J_2(x;u(\cdot)),
\end{aligned}
\]
where
\begin{equation}
\begin{aligned}
\widehat J_2(x;u(\cdot)) :=\ & \mathbb E\int_0^T\Big\{\big\langle(\widehat C^\top P\widehat C+\widehat Q)X_2,~X_2\big\rangle+2\big\langle(\widehat C^\top P \widehat D+\widehat S)u_2,~X_2\big\rangle\\
    & +\big\langle(\widehat D^\top P\widehat D+\widehat R)u_2,~u_2\big\rangle\Big\}dt+\mathbb E \big\langle \widehat
    GX_2(T),~X_2(T)\big\rangle.
\end{aligned}
\end{equation}
%

\noindent{\bf Step 2:} Now, we deal with $\widehat J_2(x;u(\cdot))$.
Note that $X_2(\cdot)=\mathbb E[X(\cdot)]$ and $u_2(\cdot)=\mathbb
E[u(\cdot)]$ are deterministic. Therefore, the LQ
problem of system \eqref{Sec2_Ex} and the cost
functional $\widehat J_2(x;u(\cdot))$ are deterministic. Let
$\widehat P(\cdot)$ be the solution of Riccati equation
\eqref{Sec2.F_Riccati_Eq2}. Differentiating $\langle\widehat
PX_2,~X_2\rangle$ and integrating from $0$ to $T$, we have
\begin{equation}\label{hatPX2}
    \begin{aligned}
        &~~~\big\langle\widehat GX_2(T),~X_2(T)\big\rangle-\big\langle\widehat P(0)X_2(0),~X_2(0)\big\rangle\\
        &=\int_0^T\Big\{\big\langle(\dot {\widehat P} +\widehat P \widehat A +\widehat A^\top \widehat P)X_2,~X_2\big\rangle+2\big\langle\widehat P\widehat Bu_2,~X_2\big\rangle\Big\}dt.
    \end{aligned}
\end{equation}
Adding \eqref{hatPX2} to $\widehat J_2(u(\cdot))$, we have
\begin{equation*}
    \begin{aligned}
        \widehat J_2(x;u(\cdot))&=\int_0^T\Big\{\big\langle(\dot {\widehat P} +\widehat P \widehat A +\widehat A^\top \widehat P+\widehat C^\top P\widehat C+\widehat Q)X_2,~X_2\big\rangle+2\big\langle(\widehat P\widehat B+\widehat C^\top P \widehat D+\widehat S)u_2,~X_2\big\rangle\\
        &~~~~~~~~~+\big\langle(\widehat D^\top P\widehat D+\widehat R)u_2,~u_2\big\rangle\Big\}dt+\big\langle\widehat P(0)x,~x\big\rangle.
    \end{aligned}
\end{equation*}
By completing the square, we reduce $\widehat J_2(x;u(\cdot))$ to
\begin{equation*}
    \begin{aligned}
        \widehat J_2(x;u(\cdot))&=\big\langle\widehat P(0)x,~x\big\rangle+\int_0^T\Big\{\big\langle (\widehat D^\top P\widehat D+\widehat R)[u_2-\widehat\Gamma(P,\widehat P)X_2],~[u_2-\widehat\Gamma(P,\widehat P)X_2]\big\rangle\\
        &~~~+\big\langle\big[\dot {\widehat P} +\widehat P \widehat A +\widehat A^\top \widehat P +\widehat C^\top P
\widehat C +\widehat Q -\widehat \Gamma(P, \widehat P)^\top \big( \widehat
D^\top P \widehat D +\widehat R \big) \widehat\Gamma(P, \widehat P)\big]X_2,~X_2\big\rangle\Big\}dt\\
&~~~ = \big\langle\widehat P(0)x,~x\big\rangle+\int_0^T\big\langle (\widehat D^\top P\widehat D+\widehat R)[u_2-\widehat\Gamma(P,\widehat P)X_2],~[u_2-\widehat\Gamma(P,\widehat P)X_2]\big\rangle dt.\\
    \end{aligned}
\end{equation*}
In summary, since $(P(\cdot), \widehat P(\cdot))$ is the solution of Riccati equations \eqref{Sec2.F_Riccati_Eq1}-\eqref{Sec2.F_Riccati_Eq2}, we have
\begin{equation}
    \begin{aligned}\label{Ju1u2}J(x;u(\cdot))&=J(x;u_1(\cdot),u_2(\cdot))\\&=\mathbb E\int_0^T\Big\{\big\langle (D^\top PD+R)[u_1-\Gamma(P)X_1],~[u_1-\Gamma(P)X_1]\big\rangle\Big\}dt\\
        &~~~+\int_0^T\Big\{\big\langle (\widehat D^\top P\widehat D+\widehat R)[u_2-\widehat\Gamma(P,\widehat P)X_2],~[u_2-\widehat\Gamma(P,\widehat P)X_2]\big\rangle\Big\}dt\\
        &~~~+\big\langle\widehat P(0)x,~x\big\rangle\geq\big\langle\widehat P(0)x,~x\big\rangle.
    \end{aligned}
\end{equation}
If we take
\begin{equation*}
    u^*_1=\Gamma(P)X^*_1,~~~~u^*_2=\widehat\Gamma(P,\widehat P)X^*_2,
\end{equation*}
where $X_1^*$ and $X_2^*$ are determined by
\begin{equation*}
\left\{\begin{aligned}
& dX_1^* =  \Big( AX_1^*+B\Gamma(P)X_1^* \Big) dt+\Big(CX_1^*+D\Gamma(P)X_1^*+ \widehat CX_2^* +\widehat D\widehat\Gamma(P,\widehat P)X_2^*
\Big)dW(t),\quad t\in [0,T],\\
& X_1^*(0) = 0,
\end{aligned}\right.
\end{equation*}
and
\begin{equation*}
\left\{\begin{aligned}
& dX_2^* =  \Big( \widehat A+\widehat B\widehat \Gamma(P,\widehat P)\Big) X_2^* dt,\quad t\in [0,T],\\
& X_2^*(0) = x,
\end{aligned}\right.
\end{equation*}
then the equality of \eqref{Ju1u2} holds. Hence, we get
\[
J(x;u^*_1(\cdot),u^*_2(\cdot))=\big\langle\widehat P(0)x,~x\big\rangle.
\]
Also, we have the following optimal control
\begin{equation*}
\begin{aligned} u^*&=u_1^*+u_2^*=\Gamma(P)X_1^*+\widehat \Gamma(P,\widehat P)X_2^*\\
    &=\Gamma(P)(X^*-\mathbb E[X^*])+\widehat \Gamma(P,\widehat P)\mathbb E[X^*],
    \end{aligned}
\end{equation*}
where $X^*(\cdot)$ is determined by \eqref{Sec2.F_Optimal_State}. Therefore,
we have
\begin{equation*}
     J\big(x;u^*(\cdot) \big)= \inf_{u(\cdot)\in \mathscr{U}[0,T]}
J\big(x;u(\cdot)\big)=\big\langle\widehat P(0)x,x\big\rangle,
\end{equation*}
which implies the desired result.
\end{proof}

\begin{proposition}\label{pro34} \sl
Let
\begin{equation}\label{Sec2.F_Relat_YZ_Xu}
\left\{
\begin{aligned}
& Y = P \big( X^* -\mathbb E[X^*] \big) +\widehat P \mathbb E[X^*],\\
& Z = P \big( CX^* +\widetilde C \mathbb E[X^*] +Du^* +\widetilde
D \mathbb E[u^*] \big),
\end{aligned}
\right. \qquad t\in [0,T],
\end{equation}
where $(X^*(\cdot),u^*(\cdot))$ defined by
\eqref{Sec2.F_Optimal_State}-\eqref{Sec2.F_Optimal_Control} and
$(P(\cdot), \widehat P(\cdot))$ is the solution of Riccati equations
\eqref{Sec2.F_Riccati_Eq1}-\eqref{Sec2.F_Riccati_Eq2}. Then
$\Theta^*(\cdot) = (X^*(\cdot), u^*(\cdot), Y(\cdot),Z(\cdot))$
defined by \eqref{Sec2.F_Optimal_Control},
\eqref{Sec2.F_Optimal_State} and \eqref{Sec2.F_Relat_YZ_Xu} is a
solution to the Hamiltonian system \eqref{Sec2.O_Hamil_Sys}.

\end{proposition}

\begin{proof}
    Firstly, it is clear that $(X^*(\cdot),u^*(\cdot))$ solves the
forward SDE (with the initial condition) in
\eqref{Sec2.O_Hamil_Sys}. Secondly, applying It\^{o}'s formula to
$Y(\cdot) = P(\cdot) \big( X^*(\cdot) -\mathbb E[X^*(\cdot)] \big)
+\widehat P(\cdot) \mathbb E[X^*(\cdot)]$, by the definition of
$Z(\cdot)$ and $u^*(\cdot)$, we have
\[
\begin{aligned}
dY =\ & d\big(P\big(X^* -\mathbb E[X^*]\big)\big)
+d\big(\widehat P\mathbb E[X^*]\big)\\
=\ & - \bigg\{ \Big[ A^\top P +C^\top PC +Q\Big] \big(X^*-\mathbb
E[X^*]\big) +\Big[ C^\top PD +S\Big] \big(u^*-\mathbb E[u^*]
\big)\\
& \qquad +\Big[\widehat A^\top\widehat P +\widehat C^\top P\widehat C
+\widehat Q\Big] \mathbb E[X^*] +\Big[ \widehat C^\top P \widehat
D +\widehat S \Big] \mathbb E[u^*] \bigg\} dt +ZdW(t) \\
=\ & -\Big\{ Q\big( X^*-\mathbb E[X^*] \big) +\widehat Q\mathbb
E[X^*] +S\big( u^*-\mathbb E[u^*]\big) +\widehat S\mathbb
E[u^*] + A^\top\big( Y-\mathbb E[Y]\big) \\
& \qquad +\widehat A^\top \mathbb E[Y] +C^\top \big(Z-\mathbb
E[Z]\big) +\widehat C^\top \mathbb E[Z] \Big\} dt +ZdW(t).
\end{aligned}
\]
Due to the definition of $g(\cdot,\Theta^*(\cdot),\mathbb
E[\Theta^*(\cdot)])$, we verify that $\Theta^*(\cdot)$
satisfies the BSDE (with the terminal condition) in the Hamiltonian
system \eqref{Sec2.O_Hamil_Sys}. Finally,
substituting \eqref{Sec2.F_Relat_YZ_Xu} into
$\Psi(\cdot,\Theta^*(\cdot),\mathbb E[\Theta^*(\cdot)])$ yields
\[
\begin{aligned}
\Psi\big(\Theta^*,\mathbb E[\Theta^*]\big) =\ & \big[ PB +C^\top
PD +S \big]^\top \big( X^* -\mathbb E[X^*] \big) +\big[ \widehat P
\widehat B +\widehat C^\top P\widehat D +\widehat S \big]^\top
\mathbb E[X^*]\\
& +\big[ D^\top P D +R \big] \big( u^* -\mathbb E[u^*] \big)
+\big[ \widehat D^\top P\widehat D +\widehat R \big] \mathbb
E[u^*].
\end{aligned}
\]
From the definition of $u^*(\cdot)$ (see
\eqref{Sec2.F_Optimal_Control}), we obtain
$\Psi(\cdot,\Theta^*(\cdot),\mathbb E[\Theta^*(\cdot)]) =0$, i.e.,
the stationarity condition in \eqref{Sec2.O_Hamil_Sys} is satisfied.
In summary, we prove that $\Theta^*(\cdot)$ is a solution to
the Hamiltonian system \eqref{Sec2.O_Hamil_Sys}.
\end{proof}

Under positive definite condition, Yong \cite{Yong-2013} studied the MF-LQ problem without cross-terms in the cost functional, and we study the case with cross-terms.
 Some of the above results of positive definite case can also be obtained by the direct method introduced in Duncan and Pasik-Duncan \cite{Duncan-Duncan-2017}, which is used to further develop a kind of nonlinear
nonquadratic mean-field type game with cross-terms in   Barreiro-Gomez et al. \cite{Barreiro-Duncan-Duncan-Tembine-2020}.

\section{Relaxed compensators and Problem (MF-LQ) in the indefinite case}\label{sec4}

In this section, we are concerned about Problem (MF-LQ) without
Condition (PD). For this indefinite case, inspired by the works of
Yu \cite{Yong-2013} and Huang and Yu
\cite{Huang-Yu-2014}, we introduce a notion named {\it relaxed
compensator} to assist our analysis.

\medskip

In detail, we introduce a space:
\[
\Lambda[0,T] = \bigg\{ F(\cdot)\ \bigg|\ F(t) = F(0) +\int_0^t
f(s)ds,\ t\in [0,T], \mbox{ where } f(\cdot)\in L^\infty(0,T;\mathbb
S^n) \bigg\}.
\]
For a given pair of functions $(H(\cdot),K(\cdot)) \in
\Lambda[0,T]\times \Lambda[0,T]$, we define (for simplicity of
notation, the argument $t$ is suppressed)
\begin{equation}\label{Sec3_Notations_bf_4HK}
\left\{
\begin{aligned}
& \mathbf Q^{H,K} = \left( \begin{array}{ccc} Q^{H,K} & O\\ O &
\widehat Q^{H,K}
\end{array} \right),\quad \mathbf S^{H,K} = \left( \begin{array}{ccc} S^{H,K} & O\\ O &
\widehat S^{H,K}
\end{array} \right),\\
& \mathbf R^{H,K} = \left( \begin{array}{ccc} R^{H,K} & O\\ O &
\widehat R^{H,K}
\end{array} \right),\quad \mathbf G^{H,K} = \left( \begin{array}{ccc} G^{H,K} & O\\ O &
\widehat G^{H,K}
\end{array} \right),
\end{aligned}
\right.
\end{equation}
where
\begin{equation}\label{Sec3_Notations_8HK}
\left\{
\begin{aligned}
& Q^{H,K} = \dot H +HA +A^\top H +C^\top HC +Q,\quad \widehat
Q^{H,K} = \dot K +K\widehat A +\widehat A^\top K
+\widehat C^\top H\widehat C +\widehat Q,\\
& S^{H,K} = HB +C^\top HD +S, \hskip 26.7mm \widehat S^{H,K} =
K\widehat B
+\widehat C^\top H\widehat D +\widehat S,\\
& R^{H,K} = D^\top HD +R,\hskip 36.4mm \widehat R^{H,K} =
\widehat D^\top H \widehat D +\widehat R,\\
& G^{H,K} = G -H(T), \hskip 39.4mm \widehat G^{H,K} = \widehat G
-K(T).
\end{aligned}
\right.
\end{equation}
According to the notation  given by \eqref{Sec3_Notations_8HK},  we introduce
\begin{equation*}
\begin{aligned}
& J^{H,K}\big(x;u(\cdot)\big) \\
=\ & \mathbb E \bigg\{ \int_0^T
\Big[ \big\langle  Q^{H,K}(t)(X(t)-\mathbb E[X(t)]),X(t)-\mathbb
E[X(t)] \big\rangle +\big\langle \widehat{
Q}^{H,K}(t)\mathbb E [X(t)],\ \mathbb E [X(t)] \big\rangle\\
& ~~~~~~~~+2\big\langle S^{H,K}(t)( u(t)-\mathbb E[ u(t)]), X(t)-\mathbb E[ X(t)]\big\rangle +2\big\langle \widehat
S^{H,K}(t)\mathbb E [u(t)],\ \mathbb E [X(t)] \big\rangle\\
& ~~~~~~~~+\big\langle R^{H,K}(t)( u(t)-\mathbb E[ u(t)]), u(t)-\mathbb E[ u(t)]\big\rangle +\big\langle \widehat
R^{H,K}(t)\mathbb E [ u(t)],\ \mathbb E [u(t)] \big\rangle \Big] dt\\
& ~~~~~~~~+\big\langle G^{H,K}(X(T)-\mathbb E[X(T)], X(T)-\mathbb E[X(T)] \big\rangle +\big\langle \widehat G^{H,K}
\mathbb E [X(T)],\ \mathbb E [X(T)] \big\rangle \bigg\}.
\end{aligned}
\end{equation*}

Then, similar to Problem (MF-LQ), we propose another MF-LQ stochastic optimal control problems as follows:

\medskip

\noindent{\bf Problem (MF-LQ)$^{H,K}$.} For given $x\in \mathbb R^n$, the problem is to find an admissible control $u^{H,K}(\cdot)
\in \mathscr U[0,T]$ such that
\begin{equation*}
J^{H,K}\big(x;u^{H,K}(\cdot) \big) = \inf_{u(\cdot)\in\mathscr{U}[0,T]} J^{H,K}\big(x;u(\cdot) \big).
\end{equation*}

\medskip

The next lemma shows the equivalence between $J(x;u(\cdot))$ and
$J^{H,K}(x;u(\cdot))$, which plays a key role in our analysis.

\begin{lemma}\label{Sec3_Lem_Cost_Relation}
Let $(H(\cdot),K(\cdot))\in \Lambda[0,T]\times\Lambda[0,T]$. For any $x\in \mathbb R^n$ and any $u(\cdot)\in \mathscr
U[0,T]$,
\begin{equation}\label{Sec3_Cost_Relation}
J^{H,K}(x;u(\cdot)) = J(x;u(\cdot)) -\langle K(0)x,\ x
\rangle.
\end{equation}
\end{lemma}

\begin{proof}
Using It\^{o}'s formula to $\langle H(\cdot)(X(\cdot)-\mathbb
E[X(\cdot)]),\ X(\cdot)-\mathbb E [X(\cdot)] \rangle$ on the
interval $[0,T]$, we get
\begin{equation}\label{Sec3_H(x-Ex)}
0 = \mathbb E \bigg\{ \int_0^T \Big[ \big\langle \Delta\mathbf Q^H
\mathbf X,\ \mathbf X \big\rangle +2\big\langle \Delta \mathbf S^H
\mathbf u,\ \mathbf X \big\rangle +\big\langle \Delta\mathbf R^H
\mathbf u,\ \mathbf u \big\rangle \Big] dt +\big\langle \Delta
\mathbf G^H \mathbf X(T),\ \mathbf X(T) \big\rangle \bigg\},
\end{equation}
where
\begin{equation*}\label{Sec2_Notations_bf_Xu}\begin{aligned}
\mathbf X(t)& =\left( \begin{array}{ccc} X(t)-\mathbb E [X(t)] \\
\mathbb E [X(t)]
\end{array} \right),~~~
\mathbf u(t)=\left( \begin{array}{ccc} u(t)-\mathbb E [u(t)] \\
\mathbb E [u(t)]
\end{array} \right),~~~
 \mathbf X(0) = \mathbf x,
\end{aligned}
\end{equation*}
and
\[
\left\{
\begin{aligned}
& \Delta\mathbf Q^H = \left( \begin{array}{ccc} \dot H +HA +A^\top H
+C^\top HC & O \\ O & \widehat C^\top H \widehat C
\end{array} \right),\quad \Delta \mathbf S^H = \left( \begin{array}{ccc} HB +C^\top HD & O \\
O & \widehat C^\top H \widehat D
\end{array} \right), \\
& \Delta \mathbf R^H = \left( \begin{array}{ccc} D^\top HD & O \\
O & \widehat D^\top H \widehat D
\end{array} \right),\qquad \Delta\mathbf G^H = \left( \begin{array}{ccc} -H(T) & O \\
O & O
\end{array} \right).
\end{aligned}
\right.
\]
Similarly, applying It\^{o}'s formula to $\langle K(\cdot) \mathbb
E[X(\cdot)],\ \mathbb E [X(\cdot)] \rangle$ leads to
\begin{equation}\label{Sec3_KEx}
\begin{aligned}
-\big\langle K(0)x,\ x \big\rangle =\ & \mathbb E  \bigg\{
\int_0^T \Big[ \big\langle \Delta\mathbf Q^K \mathbf X,\ \mathbf X
\big\rangle +2\big\langle \Delta \mathbf S^K \mathbf u,\ \mathbf X
\big\rangle +\big\langle \Delta\mathbf R^K \mathbf u,\ \mathbf u
\big\rangle \Big] dt\\
& +\big\langle \Delta \mathbf G^K \mathbf X(T),\ \mathbf X(T)
\big\rangle \bigg\},
\end{aligned}
\end{equation}
where
\[
\left\{
\begin{aligned}
& \Delta\mathbf Q^K = \left( \begin{array}{ccc} O & O\\ O & \dot K
+K \widehat A +\widehat A^\top K
\end{array} \right),\qquad \Delta \mathbf S^K = \left( \begin{array}{ccc} O & O\\ O & K
\widehat B
\end{array} \right), \\
& \Delta \mathbf R^K = O, \qquad \Delta \mathbf G^K = \left(
\begin{array}{ccc} O & O\\ O & -K(T)
\end{array} \right).
\end{aligned}
\right.
\]
Now, by the definition of $(\mathbf Q^{H,K}(\cdot), \mathbf
S^{H,K}(\cdot), \mathbf R^{H,K}(\cdot), \mathbf G^{H,K})$ (see
\eqref{Sec3_Notations_bf_4HK} and \eqref{Sec3_Notations_8HK}),
adding \eqref{Sec3_H(x-Ex)} and \eqref{Sec3_KEx} on both sides of
\eqref{Sec2_Cost_xEx} yields
\[
\begin{aligned}
& J\big(x;u(\cdot) \big) -\big\langle K(0)x,\ x \big\rangle\\
=\ & \mathbb E  \bigg\{ \int_0^T \Big[ \big\langle \big(\mathbf Q
+\Delta\mathbf Q^H +\Delta\mathbf Q^K\big) \mathbf X,\ \mathbf X
\big\rangle +2\big\langle \big(\mathbf S +\Delta\mathbf S^H
+\Delta\mathbf S^K\big) \mathbf u,\ \mathbf X \big\rangle\\
& +\big\langle \big(\mathbf R +\Delta\mathbf R^H +\Delta\mathbf
R^K\big) \mathbf u,\ \mathbf u \big\rangle \Big] dt +\big\langle
\big(\mathbf G +\Delta\mathbf G^H +\Delta\mathbf G^K\big) \mathbf
X(T),\ \mathbf X(T) \big\rangle \bigg\}\\
=\ & J^{H,K}(x;u(\cdot)).
\end{aligned}
\]
The equation \eqref{Sec3_Cost_Relation} is obtained.
\end{proof}

\medskip

\begin{definition}\label{Sec3_Def_RC} \sl
If there exists a pair of functions $(H(\cdot),K(\cdot)) \in
\Lambda[0,T] \times \Lambda[0,T]$ such that the quadruple of
functions $(\mathbf Q^{H,K}(\cdot), \mathbf S^{H,K}(\cdot), \mathbf
R^{H,K}(\cdot), \mathbf G^{H,K})$ satisfies Condition (PD), then we
call $(H(\cdot), K(\cdot))$ a relaxed compensator for Problem (MF-LQ).
\end{definition}

\begin{corollary}\label{coro} \sl
If there exists a relaxed compensator for Problem (MF-LQ), then
Problem (MF-LQ) is well-posed.
\end{corollary}

\begin{proof}
Let $(H(\cdot),K(\cdot))$ be a relaxed compensator. By the
definition, the quadruple $(\mathbf Q^{H,K}(\cdot)$, $\mathbf
S^{H,K}(\cdot), \mathbf R^{H,K}(\cdot), \mathbf G^{H,K})$ satisfies
Condition (PD). Then, for the given $x$ and any
$u(\cdot)\in \mathscr U[0,T]$, Remark \ref{Sec2_Rem} and Lemma
\ref{Sec3_Lem_Cost_Relation} imply
\[
J(x;u(\cdot)) = J^{H,K}\big(x;u(\cdot)\big) +\langle
K(0)x,\ x \rangle \geq \langle K(0)x,\ x \rangle.
\]
The conclusion is obtained.
\end{proof}

\medskip

Now we extend the solvability results (see Theorem
\ref{Sec2.O_THM_Solution} and Theorem \ref{Sec2.F_THM_Solution}) of
Problem (MF-LQ) from the positive definite case to the indefinite
case. Similar to \eqref{Sec2.O_gPsi}, for  any
$\theta = (x,u,y,z)$, $\bar\theta = (\bar x,\bar
u,\bar y,\bar z) \in \mathbb R^{n+m+n+n}$, we define
\begin{equation}
\left\{
\begin{aligned}
& g^{H,K}(r,\theta,\bar\theta) = Q^{H,K}(t)x +\big(\widehat
Q^{H,K}(t) -Q^{H,K}(t)\big)\bar x +S^{H,K}(t)u\\
& \qquad +\big(\widehat S^{H,K}(t) -S^{H,K}(t)\big) \bar u +
A(t)^\top y +\widetilde A(t)^\top \bar y +C^\top(t) z
+\widetilde C(t)^\top \bar
z,\\
& \Psi^{H,K}(r,\theta,\bar\theta) = \big(S^{H,K}(t)\big)^\top
x +\big(\widehat S^{H,K}(t)
-S^{H,K}(t)\big)^\top \bar x + R^{H,K}(t)u \\
& \qquad +\big( \widehat R^{H,K}(t) -R^{H,K}(t) \big)\bar u
+B(t)^\top y +\widetilde B(t)^\top \bar y +D(t)^\top z
+\widetilde D(t)^\top \bar z.
\end{aligned}
\right.
\end{equation}
Instead of \eqref{Sec2.O_Hamil_Sys}, the Hamiltonian system related
to Problem (MF-LQ)$^{H,K}$  is given by
\begin{equation}\label{Sec3.O_Hamil_Sys_HK}
\left\{
\begin{aligned}
& 0 = \Psi^{H,K}\big( \Theta^{H,K}, \mathbb E [\Theta^{H,K}] \big),\quad t\in [0,T],\\
& dX^{H,K} = \Big\{ AX^{H,K} +\widetilde A\mathbb E [X^{H,K}]
+Bu^{H,K} +\widetilde B \mathbb E [u^{H,K}] \Big\} dt\\
& \qquad +\Big\{ CX^{H,K} +\widetilde C\mathbb E [X^{H,K}]
+Du^{H,K} +\widetilde D \mathbb E [u^{H,K}] \Big\} dW,\quad t\in
[0,T],\\
& dY^{H,K} = -g^{H,K} \big( \Theta^{H,K}, \mathbb E [\Theta^{H,K}] \big) dt +Z^{H,K} dW,\quad t\in [0,T],\\
& X^{H,K}(0) = x,\quad Y^{H,K}(T) = G^{H,K}X^{H,K}(T)
+\big(\widehat G^{H,K} -G^{H,K} \big) \mathbb E [X^{H,K}(T)].
\end{aligned}
\right.
\end{equation}

\begin{theorem}\label{th Hamil} \sl
If there exists a relaxed compensator $(H(\cdot),K(\cdot)) \in
\Lambda [0,T] \times \Lambda [0,T]$, then for any initial state
$x$, the Hamiltonian system
\eqref{Sec2.O_Hamil_Sys} admits a unique solution $\Theta^*(\cdot)
\in M^2_{\mathbb F}(0,T)$. Moreover, $(X^*(\cdot),u^*(\cdot))$ is
the unique optimal pair of Problem (MF-LQ).
\end{theorem}

\begin{proof}
Firstly, for any given $x\in\mathbb R^n$, we prove the
equivalent unique solvability between the Hamiltonian systems
\eqref{Sec2.O_Hamil_Sys} and \eqref{Sec3.O_Hamil_Sys_HK}. In fact,
on the one hand, if $\Theta^* (\cdot) =
(X^*(\cdot),u^*(\cdot),Y(\cdot),Z(\cdot))$ is a solution to
\eqref{Sec2.O_Hamil_Sys}, then a straightforward calculation leads to
\begin{equation}\label{Sec3.O_Trans_Hamil}
\left\{
\begin{aligned}
& X^{H,K} = X^*,\qquad u^{H,K} =u^*,\\
& Y^{H,K} = Y -H\big( X^* -\mathbb E [X^*] \big) -K\mathbb
E[X^*],\\
& Z^{H,K} = Z -H\big( CX^* +\widetilde C\mathbb E [X^*] +Du^*
+\widetilde D\mathbb E [u^*] \big),
\end{aligned}
\right. \qquad t\in [0,T]
\end{equation}
is a solution to \eqref{Sec3.O_Hamil_Sys_HK}. On the other hand, if
$\Theta^{H,K}(\cdot) =
(X^{H,K}(\cdot),u^{H,K}(\cdot),Y^{H,K}(\cdot),Z^{H,K}(\cdot))$ is a
solution to \eqref{Sec3.O_Hamil_Sys_HK}, then due to the
invertibility, the transformation \eqref{Sec3.O_Trans_Hamil} yields
also a solution to \eqref{Sec2.O_Hamil_Sys}. Therefore, the
existence and uniqueness between \eqref{Sec2.O_Hamil_Sys} and
\eqref{Sec3.O_Hamil_Sys_HK} are equivalent.

Secondly, since $(H(\cdot),K(\cdot))$ is a relaxed compensator, by
Definition \ref{Sec3_Def_RC}, the quadruple $(\mathbf
Q^{H,K}(\cdot), \mathbf S^{H,K}(\cdot), \mathbf R^{H,K}(\cdot),
\mathbf G^{H,K})$ satisfies Condition (PD). By Theorem
\ref{Sec2.O_THM_Solution}, the stochastic Hamiltonian system
\eqref{Sec3.O_Hamil_Sys_HK} related to Problem (MF-LQ)$^{H,K}$ admits a unique solution $\Theta^{H,K}(\cdot)$.
Moreover $(X^{H,K}(\cdot), u^{H,K}(\cdot))$ is the unique optimal
pair of Problem (MF-LQ)$^{H,K}$. By the analysis in the
above paragraph, the stochastic Hamiltonian system
\eqref{Sec2.O_Hamil_Sys} related to Problem (MF-LQ)
admits also a unique solution $\Theta^*(\cdot)$. Moreover,
$(X^*(\cdot),u^*(\cdot)) = (X^{H,K}(\cdot),u^{H,K}(\cdot))$. By the
equivalence between the cost functionals $J^{H,K}(u(\cdot))$
and $J(x;u(\cdot))$ (see Lemma \ref{Sec3_Lem_Cost_Relation}),
the unique optimal pair $(X^*(\cdot),u^*(\cdot)) =
(X^{H,K}(\cdot),u^{H,K}(\cdot))$ of Problem (MF-LQ)$^{H,K}$ (which is the conclusion of Theorem
\ref{Sec2.O_THM_Solution}) is also the unique optimal pair of
Problem (MF-LQ). The proof is completed.
\end{proof}

Theorem \ref{th Hamil} solves the MF-LQ problem in  indefinite  condition. Moreover, it also gives a new condition about the solvability of MF-FBSDEs. Please see an example about MF-FBSDEs not satisfying monotonicity condtion in Section 5.2 for details.

\medskip

Next, we turn to the issue of the feedback representation for the
optimal control in the indefinite case. Similar to
\eqref{Sec2.F_Gamma}, we define $\Gamma^{H,K}: [0,T]\times\mathbb
S^n \rightarrow \mathbb R^{m\times n}$ and $\widehat\Gamma^{H,K}:
[0,T] \times \mathbb S^n \times \mathbb S^n \rightarrow \mathbb
R^{m\times n}$ as follows:
\begin{equation*}
\left\{
\begin{aligned}
& \Gamma^{H,K}\big(t, P\big) = - \big[D(t)^\top PD(t)
+R^{H,K}(t)\big]^{-1} \big[PB(t) +C(t)^\top
PD(t) +S^{H,K}(t)\big]^\top,\\
& \widehat\Gamma^{H,K}\big(t, P, \widehat P\big) = - \big[ \widehat
D(t)^\top P \widehat D(t) +\widehat R^{H,K}(t) \big]^{-1} \big[ \widehat P
\widehat B(t) +\widehat C(t)^\top P \widehat D(t) +\widehat
S^{H,K}(t) \big]^\top.
\end{aligned}
\right.
\end{equation*}
Then the system of Riccati equations related to Problem
(MF-LQ)$^{H,K}$ is given by
\begin{equation}\label{Sec3.F_Riccati_HK_Eq1}
\left\{
\begin{aligned}
& \dot P^{H,K} +P^{H,K}A +A^\top P^{H,K} +C^\top P^{H,K}C +Q^{H,K}\\
& \qquad -\Gamma^{H,K}(P^{H,K})^\top \big[D^\top P^{H,K}D +R^{H,K}\big] \Gamma^{H,K}(P^{H,K}) = 0,\quad t\in [0,T],\\
& P^{H,K}(T) = G^{H,K},\\
& D^\top P^{H,K}D +R^{H,K} \gg 0, \qquad t\in [0,T]
\end{aligned}
\right.
\end{equation}
and
\begin{equation}\label{Sec3.F_Riccati_HK_Eq2}
\left\{
\begin{aligned}
& \dot {\widehat P}^{H,K} +\widehat P^{H,K} \widehat A +\widehat A^\top \widehat P^{H,K}
+\widehat C^\top P^{H,K} \widehat C +\widehat Q^{H,K}\\
& \qquad -\widehat \Gamma^{H,K}(P^{H,K}, \widehat P^{H,K})^\top \big[
\widehat D^\top P^{H,K} \widehat D +\widehat R^{H,K} \big]
\widehat\Gamma^{H,K}(P^{H,K}, \widehat P^{H,K})
= 0,\quad t\in [0,T],\\
& \widehat P^{H,K} (T) = \widehat G^{H,K}, \\
& \widehat D^\top P^{H,K} \widehat D + \widehat R^{H,K}  \gg 0, \qquad t\in [0,T].
\end{aligned}
\right.
\end{equation}

\medskip

\begin{theorem}\label{Th Rel} \sl
If there exists a relaxed compensator $(H(\cdot),K(\cdot)) \in
\Lambda [0,T] \times \Lambda[0,T]$, then the system of Riccati
equations \eqref{Sec2.F_Riccati_Eq1} and \eqref{Sec2.F_Riccati_Eq2}
admits a unique pair of solutions $(P(\cdot), \widehat P(\cdot))$ taking
values in $\mathbb S^n \times \mathbb S^n$. Moreover for the initial
state $x\in \mathbb R^n$, the unique optimal pair $(X^*(\cdot),
u^*(\cdot))$ of Problem (MF-LQ) admits the
feedback form given by \eqref{Sec2.F_Optimal_Control} and
\eqref{Sec2.F_Optimal_State}.
\end{theorem}

\begin{proof}
Firstly, we prove the equivalent unique solvability between the
system of Riccati equations
\eqref{Sec2.F_Riccati_Eq1}-\eqref{Sec2.F_Riccati_Eq2} and
\eqref{Sec3.F_Riccati_HK_Eq1}-\eqref{Sec3.F_Riccati_HK_Eq2}. In
fact, on one hand, if $(P(\cdot),\widehat P(\cdot))$ taking values in
$\mathbb S^n \times \mathbb S^n$ is a solution to
\eqref{Sec2.F_Riccati_Eq1}-\eqref{Sec2.F_Riccati_Eq2}, then by a
straightforward calculation,
\begin{equation}\label{Sec3.F_Trans_Riccati}
P^{H,K}(\cdot) = P(\cdot) -H(\cdot),\qquad \widehat P^{H,K}(\cdot) =
\widehat P(\cdot) -K(\cdot)
\end{equation}
is a solution to
\eqref{Sec3.F_Riccati_HK_Eq1}-\eqref{Sec3.F_Riccati_HK_Eq2}. On the
other hand, if $(P^{H,K}(\cdot), \widehat P^{H,K}(\cdot))$ taking values in
$\mathbb S^n \times\mathbb S^n$ is a solution to
\eqref{Sec3.F_Riccati_HK_Eq1}-\eqref{Sec3.F_Riccati_HK_Eq2}, then
the inverse transformation of \eqref{Sec3.F_Trans_Riccati} provides
a solution to \eqref{Sec2.F_Riccati_Eq1}-\eqref{Sec2.F_Riccati_Eq2}.
Therefore, the existence and uniqueness between
\eqref{Sec2.F_Riccati_Eq1}-\eqref{Sec2.F_Riccati_Eq2} and
\eqref{Sec3.F_Riccati_HK_Eq1}-\eqref{Sec3.F_Riccati_HK_Eq2} are
equivalent.

\medskip

Since $(H(\cdot),K(\cdot))$ is a relaxed compensator, then the
quadruple $(\mathbf Q^{H,K}(\cdot), \mathbf S^{H,K}(\cdot), \mathbf
R^{H,K}(\cdot)$, $\mathbf G^{H,K} )$ satisfies Condition (PD). By
Theorem \ref{Sec2.F_THM_Solution}, the system of Riccati equations
\eqref{Sec3.F_Riccati_HK_Eq1}-\eqref{Sec3.F_Riccati_HK_Eq2} admits a
unique solution. By the analysis in the previous paragraph, the same
is true for the system
\eqref{Sec2.F_Riccati_Eq1}-\eqref{Sec2.F_Riccati_Eq2}.

\medskip

Let
\begin{equation}\label{Sec3.F_Optimal_Control_HK}
u^{H,K} = \Gamma^{H,K}\big(P^{H,K}\big) \big( X^{H,K} -\mathbb
E[X^{H,K}] \big) +\widehat\Gamma^{H,K}\big(P^{H,K},\Pi^{H,K}\big)
\mathbb E [X^{H,K}],\quad t \in [0,T],
\end{equation}
where $X^{H,K}(\cdot)$ satisfies
\begin{equation}\label{Sec3.F_Optimal_State_HK}
\left\{
\begin{aligned}
& dX^{H,K} = \Big\{ \big( A+B\Gamma^{H,K}(P^{H,K}) \big) \big(
X^{H,K} -\mathbb E [X^{H,K}] \big)\\
& \qquad  +\big( \widehat A +\widehat
B\widehat\Gamma^{H,K}(P^{H,K},\Pi^{H,K}) \big) \Big\} dt +\Big\{
\big( C+D\Gamma^{H,K}(P^{H,K}) \big) \big( X^{H,K} -\mathbb
E[X^{H,K}] \big)\\
& \qquad +\big( \widehat C +\widehat
D\widehat\Gamma^{H,K}(P^{H,K},\Pi^{H,K})
\big) \Big\} dW,\quad t\in [0,T],\\
& X^{H,K}(0) = x.
\end{aligned}
\right.
\end{equation}
Theorem \ref{Sec2.F_THM_Solution} implies that the admissible pair
$(X^{H,K}(\cdot),u^{H,K}(\cdot))$ is optimal for Problem
(MF-LQ)$^{H,K}$. It is easy to verify that
\[
\Gamma\big( P \big) = \Gamma^{H,K}\big( P^{H,K} \big), \qquad
\widehat\Gamma\big( P, \widehat P \big) = \widehat\Gamma^{H,K} \big(
P^{H,K}, \widehat P^{H,K} \big).
\]
Therefore, the admissible pair $(X^*(\cdot),u^*(\cdot))$ defined by
\eqref{Sec2.F_Optimal_Control}-\eqref{Sec2.F_Optimal_State} is the
same as $(X^{H,K}(\cdot), u^{H,K}(\cdot))$ defined by
\eqref{Sec3.F_Optimal_Control_HK}-\eqref{Sec3.F_Optimal_State_HK}.
By Lemma \ref{Sec3_Lem_Cost_Relation}, the unique optimal pair
$(X^*(\cdot),u^*(\cdot)) = (X^{H,K}(\cdot),u^{H,K}(\cdot))$ of
Problem (MF-LQ)$^{H,K}$ is also the unique optimal pair
of Problem (MF-LQ). The proof is completed.
\end{proof}

If there exist nonhomogeneous terms in system \eqref{Sec2_Sys} and linear terms in cost functional \eqref{Sec2_Cost}, these terms do not affect the well-posedness of Problem (MF-LQ). We can parallely derive the corresponding results similar to the theoretical ones established in this paper. For example, consider $J(x;u(\cdot))$  in the form of \eqref{Sec2_Cost} plus a linear term $\langle g,~\mathbb E [X(T)]\rangle$ with an $n$-dimensional constant vector $g$ as 
\begin{equation}\label{tilde J}
\widetilde J(x;u(\cdot))=J(x;u(\cdot))+2\langle g,~\mathbb E [X(T)]\rangle.
\end{equation}
The similar results can be parallely obtained. We present the following corollary as one example in details. For convenience, we call MF-LQ problem with respect to system \eqref{Sec2_Sys} and cost functional \eqref{tilde J} as {\bf Problem (MF-LQ)$^L$} and denote $\mathcal D(\varphi)=-\big[ \widehat D^\top P
\widehat D +\widehat R \big]^{-1}  \widehat B^\top \varphi$. 

\begin{corollary}\label{corollary-linear} \sl
If there exists a relaxed compensator $(H(\cdot),K(\cdot)) \in
\Lambda [0,T] \times \Lambda[0,T]$, Problem (MF-LQ)$^L$  admits  
a unique optimal feedback control
 \begin{equation*}\label{Sec2.F_Optimal_Control-linear}
 \begin{aligned}
u^* &= \Gamma\big(P\big) \big( X^* -\mathbb E [X^*] \big)
+\widehat\Gamma\big(P,\widehat P\big) \mathbb E [X^*]+\mathcal D(\varphi)
 \end{aligned}
\end{equation*}
with $\varphi$ satisfying 
\begin{equation*}
\left\{\begin{aligned}
		&\dot \varphi+\big[\widehat A+\widehat B\widehat \Gamma(P,\widehat P)\big]^\top \varphi=0,\quad t\in [0,T],\\
		&\varphi(T)=g,
	\end{aligned}
	\right.
\end{equation*}
and $(P, \widehat P)$ being the unique solution of Riccati equations \eqref{Sec2.F_Riccati_Eq1}-\eqref{Sec2.F_Riccati_Eq2}, where the optimal state $X^*(\cdot)$ satisfies 
\begin{equation*}\label{Sec2.F_Optimal_State-linear}
\left\{
\begin{aligned}
& dX^* = \Big\{ \big( A+B\Gamma(P) \big) \big( X^* -\mathbb E [X^*]\big) +\big( \widehat A +\widehat B\widehat\Gamma(P,\widehat P) \big)\mathbb E[X^*]  -\widehat B\big[ \widehat D^\top P
\widehat D +\widehat R \big]^{-1}  \widehat B^\top \varphi \Big\} dt\\
& \quad  +\Big\{ \big( C+D\Gamma(P) \big) \big( X^* -\mathbb E[X^*] \big)+\big( \widehat C +\widehat D\widehat\Gamma(P,\widehat P)\big)\mathbb E[X^*]-\widehat D\big[ \widehat D^\top P
\widehat D +\widehat R \big]^{-1}  \widehat B^\top \varphi \Big\}dW(t),\\
& X^*(0) = x.
\end{aligned}
\right.
\end{equation*}
 \end{corollary} 
Based on the result of Theorem 5.2 in \cite{Sun-2017}, this corollary can be proved, similar to the proof of Theorem \ref{Sec2.F_THM_Solution}. We omit the proof here. 

\begin{remark} \sl
When there exists a relaxed compensator $(H(\cdot),K(\cdot))\in
\Lambda [0,T] \times \Lambda[0,T]$, from
\eqref{Sec3.F_Trans_Riccati}, we can derive the following
inequalities:
\begin{equation}
H(\cdot) \leq P(\cdot),\qquad K(\cdot) \leq \widehat P(\cdot),
\end{equation}
where $(P(\cdot),\widehat P(\cdot))$ is the solution to the system of
Riccati equations.
\end{remark}

In the rest of this section, we shall propose a necessary and
sufficient condition for a relaxed compensator. For this aim, we
borrow a basic result from the theory of linear algebra.

\begin{lemma}[Schur's lemma ]\label{Sec3_Lem_Schur} \sl
Let $A\in \mathbb S^n$, $B\in \mathbb S^m$, and $C\in \mathbb
R^{n\times m}$. Then the following two statements are equivalent:
\begin{enumerate}[\rm(i).]
\item $B>0$ and $A-CB^{-1}C^\top \geq 0$;
\item $B>0$ and $\left( \begin{array}{ccc} A & C \\ C^\top & B \end{array}
\right) \geq 0$.
\end{enumerate}
\end{lemma}

\medskip

\noindent Let $(H(\cdot), K(\cdot)) \in \Lambda[0,T] \times \Lambda
[0,T]$. We introduce

\medskip

\noindent{\bf Condition (RC).} The following two groups of
inequalities hold (the argument $t$ is suppressed):
\begin{equation}\label{Sec3_RC_Ineq1}
\mbox{(i).} \  \left\{
\begin{aligned}
& \dot H +HA +A^\top H +C^\top HC +Q\\
& \quad -\big[HB +C^\top HD +S\big] \big[D^\top HD +R\big]^{-1}
\big[HB
+C^\top HD +S\big]^\top \geq 0,\quad t\in [0,T],\\
& H(T) \leq G,\\
& D^\top HD +R \gg 0, \qquad t\in [0,T]
\end{aligned}
\right.
\end{equation}
and
\begin{equation}\label{Sec3_RC_Ineq2}
\mbox{(ii).}\ \left\{
\begin{aligned}
& \dot K +K\widehat A +\widehat A^\top K
+\widehat C^\top H \widehat C +\widehat Q\\
& \quad -\big[ K\widehat B +\widehat C^\top H \widehat D +\widehat S
\big] \big[ \widehat D^\top H \widehat D +\widehat R \big]^{-1}
\big[ K\widehat B +\widehat C^\top H \widehat D +\widehat S
\big]^\top
\geq 0,\quad t\in [0,T], \\
& K(T) \leq \widehat G, \\
& \widehat D^\top H\widehat D +\widehat R \gg 0,\qquad t\in [0,T].
\end{aligned}
\right.
\end{equation}

\begin{proposition}\label{Pro} \sl
A pair of functions $(H(\cdot),K(\cdot)) \in \Lambda[0,T] \times
\Lambda[0,T]$ is a relaxed compensator for Problem (MF-LQ) if and
only if Condition (RC) holds.
\end{proposition}

\begin{proof}
By Definition \ref{Sec3_Def_RC}, $(H(\cdot),K(\cdot))$ is a relaxed
compensator if and only if Condition (PD) holds for the quadruple
$(\mathbf Q^{H,K}(\cdot), \mathbf S^{H,K}(\cdot), \mathbf
R^{H,K}(\cdot), \mathbf G^{H,K})$. By Lemma \ref{Sec3_Lem_Schur},
the first inequality in Condition (PD) is equivalent to
\[
\left\{
\begin{aligned}
& Q^{H,K} -S^{H,K} (R^{H,K})^{-1}(S^{H,K})^\top \geq 0,\\
& \widehat Q^{H,K} -\widehat S^{H,K} (\widehat
R^{H,K})^{-1}(\widehat S^{H,K})^\top \geq 0.
\end{aligned}
\right.
\]
By some straightforward calculations, we verify that Condition (PD) is equivalent to Condition (RC). The proof is completed.
\end{proof}

\begin{remark} \sl
By comparing the system of Riccati equations
\eqref{Sec2.F_Riccati_Eq1}-\eqref{Sec2.F_Riccati_Eq2} with the
system of inequalities \eqref{Sec3_RC_Ineq1}-\eqref{Sec3_RC_Ineq2}
in Condition (RC), we find the following two facts.
\begin{enumerate}[\rm(i).]
\item If the system of Riccati equations
\eqref{Sec2.F_Riccati_Eq1}-\eqref{Sec2.F_Riccati_Eq2} is solvable,
then the solution $(P(\cdot),\widehat P(\cdot))$ is a relaxed compensator
for Problem (MF-LQ). Consequently, in the indefinite case, the
solvability of the system of Riccati equations
\eqref{Sec2.F_Riccati_Eq1}-\eqref{Sec2.F_Riccati_Eq2} implies the
solvability of Problem (MF-LQ).
\item The first two equations
in \eqref{Sec2.F_Riccati_Eq1} and two equations in
\eqref{Sec2.F_Riccati_Eq2} are relaxed into the corresponding
inequalities in \eqref{Sec3_RC_Ineq1}-\eqref{Sec3_RC_Ineq2}. The solvability of the system of inequalities
\eqref{Sec3_RC_Ineq1}-\eqref{Sec3_RC_Ineq2} also implies the
solvability of Problem (MF-LQ). This can be regarded as an
explanation of the notion of relaxed compensators from the viewpoint of
Riccati equations.
\end{enumerate}
\end{remark}

Then, we present the relationship between relaxed compensator and solutions of Riccati equations by a corollary.
\begin{corollary} \sl
A relaxed compensator $(H(\cdot),K(\cdot)) \in
\Lambda [0,T] \times \Lambda[0,T]$ exists if and only if the system
of Riccati equations \eqref{Sec2.F_Riccati_Eq1} and
\eqref{Sec2.F_Riccati_Eq2} admits a unique pair of solutions
$(P(\cdot), \widehat P(\cdot))$ taking values in $\mathbb S^n \times
\mathbb S^n$.
\end{corollary}
\begin{proof}
The sufficient condition could be obtained  by Theorem \ref{Th Rel}. Next, we prove the necessary condition. If $(P(\cdot), \widehat P(\cdot))$ is the solution of Riccati
equations \eqref{Sec2.F_Riccati_Eq1} and \eqref{Sec2.F_Riccati_Eq2}, which satisfies \eqref{Sec3_RC_Ineq1} and \eqref{Sec3_RC_Ineq2}.  From Proposition \ref{Pro}, $(P(\cdot), \widehat P(\cdot))$ is a relaxed compensator.
\end{proof}

Next, we explain the effect of $K$ as a relaxed compensator.
\begin{remark} \sl
    Different from the classic LQ problem, because of the existence of the mean field item $\mathbb E [X]$ in system, $K(\cdot)$ plays a key role as one of the compensator. Now, we will explain this point.

    For simplicity, we consider the following Problem (MF-LQ) with $t$ suppressed. The system is
\begin{equation}\label{Sec3_SysK}
\left\{
\begin{aligned}
& dX= \Big\{ AX +\widetilde A \mathbb E[X]
 \Big\} dt+\Big\{Du
+\widetilde D\mathbb E [u] \Big\}dW(t), \quad t\in [0,T],\\
& X(0) = x,
\end{aligned}
\right.
\end{equation}
and the cost functional is
\begin{equation}\label{Sec3_CostK}
\begin{aligned}
J\big(x;u(\cdot)\big) =\ & \mathbb E  \int_0^T \Big[
\big\langle \widetilde
Q\mathbb E [X],\ \mathbb E [X] \big\rangle+\big\langle Ru,\ u \big\rangle +\big\langle \widetilde
R\mathbb E [u],\ \mathbb E [u] \big\rangle \Big] dt\\
& +\mathbb E \big\langle GX(T),\ X(T) \big\rangle +\big\langle \widetilde G
\mathbb E [X(T)],\ \mathbb E [X(T)] \big\rangle,
\end{aligned}
\end{equation}
where the coefficients $\widehat Q(\cdot)<0$, $R(\cdot)\leq 0$, $\widehat R(\cdot)\leq 0$, $G>0$ and $\widetilde G\in\mathbb S^n$. Obviously, this MF-LQ problem is indefinite. If there exists $(H(\cdot),K(\cdot)) \in
\Lambda [0,T] \times \Lambda[0,T]$ satisfy
\begin{equation}\label{Sec3_RC_Ineq1K}
\mbox{(i).} \  \left\{
\begin{aligned}
& \dot H +HA +A^\top H +Q \geq 0,\quad t\in [0,T], \\
& H(T) \leq G, \\
& D^\top HD +R \gg 0, \qquad t\in [0,T]
\end{aligned}
\right.
\end{equation}
and
\begin{equation}\label{Sec3_RC_Ineq2K}
\mbox{(ii).}\ \left\{
\begin{aligned}
& \dot K +K\widehat A +\widehat A^\top K +\widehat Q \geq 0,\quad t\in [0,T], \\
& K(T) \leq \widehat G, \\
& \widehat D^\top H\widehat D +\widehat R \gg 0,\qquad t\in [0,T],
\end{aligned}
\right.
\end{equation}
then $(H(\cdot),K(\cdot))$ is the relaxed compensator.
For the reason of that $H(\cdot)$ does not appear in \eqref{Sec3_RC_Ineq2K}, $H(\cdot)$ can not work on the compensation of $\widehat Q(\cdot)(<0)$, then we have to find another one: $K(\cdot)$, to compensate $\widehat Q(\cdot)$ such that this MF-LQ problem is well-posed.

For giving more details, we simplify the coefficients as constants in system \eqref{Sec3_SysK} and  \eqref{Sec3_CostK}. Although $R$ and $\widehat R$ could be negative, they can not be too negative, say, $R>-D^2G\exp\{2A(T-t)\}$ and $\widehat R>-\widehat D^2G\exp\{2A(T-t)\}$. We choose $H=G\exp\{2A(T-t)\}$ and
$K=-\widetilde Q/2\widehat A+(\widehat G+\widetilde Q/2\widehat A)\exp\{2\widehat A(T-t)\}$, by some calculations, $(H(\cdot),K(\cdot))$ satisfies conditions \eqref{Sec3_RC_Ineq1K}-\eqref{Sec3_RC_Ineq2K}, then $(H(\cdot),K(\cdot))$ is a relaxed compensator, this MF-LQ problem is well-posed.

\end{remark}

\section{Applications}\label{sec5}

\subsection{Mean-variance Portfolio Selection Problem}\label{example1}
In this subsection, a dynamic mean-variance portfolio problem is considered within the framework of indefinite MF-LQ. In the market,  we suppose that there are $m+1$ assets traded continuously under self-financing assumption. 
One asset is risk-free (for example, a default-free bond without coupons), whose price process $S_0(t)$ is governed by the following ordinary differential equation (ODE):
\begin{equation*}
\left\{\begin{aligned}
dS_0(t)&=r(t)S_0(t)dt,~~~t\in[0,T],\\
S_0(0)&=s_0,
    \end{aligned}\right.
\end{equation*}
where $s_0>0$ is the initial price and $r(\cdot)$ is nonnegative bounded function and presents the interest rate of bond. Additionally, the
other $m$ assets are securities (for example, stocks), whose price
processes $S_i(\cdot)$ ($i=1,2,\cdots,m$) satisfy the following SDE:
\begin{equation*}
\left\{\!\begin{aligned}
        dS_i(t)&=S_i(t)\Big\{\mu_i(t)dt+\sum_{j=1}^{m}\sigma_{ij}(t)dW^j(t)\Big\},  ~ t\in[0,T], \\
        S_i(0)&=s_i,
\end{aligned}\right.
\end{equation*}
where $s_i>0$ is the initial price, $\mu(\cdot):=(\mu_1(\cdot),\mu_2(\cdot),\cdots,\mu_m(\cdot))^\top$ with $\mu_i(\cdot)>0$ is
the appreciation rate, and
$\sigma_i(\cdot):=(\sigma_{i1}(\cdot),\sigma_{i2}(\cdot),\cdots$, $\sigma_{im}(\cdot))$ $(i=1,2,\cdots,m)$ is the
volatility of stocks. Define the covariance matrix
$\sigma(\cdot)
:=(\sigma_{ij}(\cdot))_{m\times m}$. Assume that $\mu(\cdot)$ and $\sigma(\cdot)$  are bounded functions. Furthermore, we assume that there exists
a constant $\delta>0$ such that
\[
\sigma(t)\sigma(t)^\top \geq \delta I,~~~\mbox{for all } t\in [0,T],
\]
where $I$ denotes the identity $m\times m$ matrix.

In financial investment, the investor's total wealth 
is denoted by $X(\cdot)$, and the amount of the  wealth invested
in the  $i$-th stock is denoted by $\pi_i(\cdot)$ ($i=1,2,\cdots,m$). Since the strategy $\pi(\cdot):=(\pi_1(\cdot),\pi_2(\cdot),\cdots,\pi_m(\cdot))^\top$ is used in a self-financing way, the
wealth invested in the bond is $X(\cdot)-\sum_{i=1}^m\pi_i(\cdot)$. Then, the wealth
process $X(\cdot)$ with the initial endowment $x$ satisfies the following SDE
\begin{equation*}
    \left\{\begin{aligned}
        dX(t)&=\big[r(t)X(t)+b(t)^\top u(t)\big]dt+u(t)^\top dW(t), \\
        X(0)&=x,
    \end{aligned}\right.
\end{equation*}
 where $x>0$ is the initial wealth, 
$u(t)=\sigma(t)^\top\pi(t)$ and
$b(t)=\sigma(t)^{-1}(\mu(t)-r(t)\mathbf 1)$ for all $t\in[0,~T]$.
Here, $\mathbf 1$ denotes the vector of all entries with $1$ and $W(\cdot)=(W^1(\cdot),W^2(\cdot),\cdots,W^m(\cdot))^\top$ is $m$-dimensional standard Brownian motion.  All the theoretical results established in this paper hold true for $m$-dimensional standard Brownian motion case.

The mean-variance problem means that the investor's objective is to maximize the expected terminal wealth $\mathbb E[X(T)]$ as well as to minimize the variance of the terminal wealth $\mbox{Var}(X(T))$.
 Let $\nu$ be a positive constant. Then, the cost functional is
\begin{equation}\label{J-MV}
    \begin{aligned}
        J(x;u(\cdot))=\frac{\nu}{2}\mbox{Var}(X(T))-\mathbb E[X(T)].
    \end{aligned}
\end{equation}

\noindent\textbf{Problem (MV). } The mean-variance portfolio selection problem is to find an
admissible control $u^*(\cdot)\in\mathscr{U}[0,T]$ satisfying $$
{J}(x;u^*(\cdot))=\inf_{u(\cdot)\in\mathscr{U}[0,T]}{J}(x;u(\cdot)).$$ 
Such an admissible control $u^*(\cdot)$ is called an optimal control,
and $X^*(\cdot)=X^*(\cdot;x, u^*(\cdot))$ is called the corresponding optimal
trajectory.

We deal with
Problem (MV) as a special case of Problem (MF-LQ)$^L$ with indefinite matrices. In this example, 
\eqref{J-MV} can be rewritten as $$
        J(x;u(\cdot))=\frac{\nu}{2}\mathbb E[X^2(T)]-\frac{\nu}{2}\big(\mathbb E[X(T)]\big)^2-\mathbb E[X(T)],$$ then, $\mathbf Q(\cdot)=\mathbf S(\cdot)=\mathbf R(\cdot)=0$, $G=\frac{\nu}{2}$, $\widetilde G=-\frac{\nu}{2}$ and $g=\frac{1}{2}$ . From Corollary \ref{corollary-linear}, we present the closed-loop form of optimal control by the
following proposition.

 \begin{proposition}\sl	
Problem (MV) admits a unique optimal control in the following closed-loop form: 
\[
\begin{aligned}
& u^*(t) = -b(t) \bigg\{ X^*(t) -\mathbb E[X^*(t)]-\frac{1}{\nu} \exp\bigg[ \int_t^T \Big( |b(s)|^2 -r(s) \Big) ds \bigg] \bigg\},\quad t\in [0,T],
\end{aligned}
\]
where $X^*(\cdot)$ satisfies
\[
\left\{
\begin{aligned}
        dX^* &=\  \Big\{\big(r(t)-|b(t)|^2)X^*+|b(t)|^2 \mathbb E[X^*]+\frac{|b(t)|^2}{\nu} \exp\bigg[ \int_t^T \Big( |b(s)|^2 -r(s) \Big) ds \bigg]\Big\}dt\\
                   &-b(t)^\top \bigg\{ X^*(t) -\mathbb E[X^*(t)]-\frac{1}{\nu} \exp\bigg[ \int_t^T \Big( |b(s)|^2 -r(s) \Big) ds \bigg] \bigg\}dW(t),~~~t\in [0,T],\\
         X(0)& =\  x.
\end{aligned}
\right.
\]
\end{proposition}

\begin{proof}
	 The corresponding Riccati equations of Problem (MV) are
\begin{equation*}
	    \left\{
        \begin{aligned}
            &~\dot P(t)+2r(t)P(t)-\vert b(t)\vert^2P(t)= 0,\quad t\in [0,T],\\
            &~P(T)= \frac{\nu}{2},
        \end{aligned}
        \right.
\end{equation*}
and
\begin{equation*}
    \left\{
        \begin{aligned}
            &~\dot {\widehat P}(t)+2r(t)\widehat P(t)-\frac{\vert b(t)\vert^2 \widehat P(t)^2}{P(t)}=0,\quad t\in [0,T],\\
            &~\widehat P(T)=0,
        \end{aligned}
        \right.
    \end{equation*}
which admit the solutions
\begin{equation}\label{PPP}
P(t)=\frac{\nu}{2}\exp\Big(\int_t^T[2r(s)-\vert b(s)\vert^2]ds\Big),\quad t\in [0,T] 
\end{equation}
and
\[
\widehat P(t)=0,\quad t\in [0,T],
\]
respectively. We choose $(P(\cdot), \widehat P(\cdot))$ as a relaxed compensator. By a direct calculation, we have $\Gamma(t,P)=-\frac{1}{P}\cdot b(t) P=-b(t)$, $\widehat\Gamma(t,P,\widehat P)=-\frac{1}{P}\cdot b(t) \widehat P=0$ and $\mathcal D(t,\varphi)=-\frac{1}{P}\cdot b(t) \varphi$, by Corollary  \ref{corollary-linear}, Problem (MV) admits a unique optimal control:
\begin{equation}\label{uuu}
u^*(t)=-b(t)\Big[X^*(t)-\mathbb E[X^*(t)]+\frac{\varphi(t)}{P(t)}\Big],\quad t\in [0,T], 
\end{equation}
where $\varphi(\cdot)$ is the solution to
\begin{equation*}\label{varphi}
\left\{
\begin{aligned}
& \dot \varphi(t)+r(t)\varphi(t)=0,\quad t\in [0,T],\\
& \varphi(T)=-\frac{1}{2}.
\end{aligned}
\right.
\end{equation*}
Explicitly,
\begin{equation}\label{phi}
\varphi(t) = -\frac{1}{2} \exp\bigg\{ \int_t^T  r(s)ds \bigg\},\quad t\in [0,T]. 	
\end{equation}

Substituting \eqref{PPP} and \eqref{phi} into \eqref{uuu} leads to the desired result. ~\hfill $\Box$
\end{proof}

\subsection{An  Example about Problem (MF-LQ)}\label{example2}

In this part, we consider an example about Problem (MF-LQ). In this example, we not only obtain the optimal control, but also obtain the unique solvability of a kind of MF-FBSDE not satisfying the monotonicity condition in \cite{Bensoussan-Yam-Zhang-2015}. Consider the following system 
\begin{equation*}\label{ex2 sys}
\left\{\begin{aligned}
        dX(t)&=\big\{a(t)X(t)+\tilde a(t)\mathbb E[X(t)]+b(t)u(t)+\tilde b(t)\mathbb E[u(t)]\big\}dt+u(t)dW(t),\quad t\in [0,T],\\
        X(0)&=x,
    \end{aligned}\right.
\end{equation*}
and the cost functional
\begin{equation*}
    \begin{aligned}
        J(x;u(\cdot))&=\mathbb E\int_0^T\Big\{\alpha \big|X(t)-\mathbb E[X(t)]\big|^2-\beta |u(t)|^2\Big\}dt+\gamma\mathbb E[X^2(T)],
    \end{aligned}
\end{equation*}
where $x\in\mathbb R$, $a(t)$, $\tilde a(t)$, $b(t)$, $\tilde b(t)$ are $1$-dimensional deterministic functions, and $\alpha$, $\beta$, $\gamma$ are  constants. The coefficients satisfy $\alpha\geq0$, $\gamma>\max\{\beta,0\}$ and $a(t) \geq (b(t)^2 \gamma)/(2(\gamma -\beta))$. Denote $\hat a(s)=a(s)+\tilde a(s)$ and $\hat b(s)=b(s)+\tilde b(s)$. The objective of this  problem is to find an admissible
control $u^*(\cdot) \in \mathscr{U}[0,T]$ such that
\begin{equation*}
J\big(x;u^*(\cdot)\big) = \inf_{u(\cdot)\in \mathscr{U}[0,T]}J\big(x;u(\cdot) \big).
\end{equation*}
  When $\beta<0$, this MF-LQ problem is under positive definite case, no more tautology here. We mainly discuss the indefinite case. We verify that $H(t) = \gamma$ and 
\[
\begin{aligned}
K(t) =\ & \bigg[ \frac 1 \gamma \exp\bigg\{ -2\int_t^T \hat a(s) ds \bigg\}+\int_t^T \frac{\hat b(s)^2}{\gamma-\beta} \exp\bigg\{ -2\int_t^s \hat a(\tau)d\tau \bigg\} ds \bigg]^{-1}
\end{aligned}
\]
constitute a relaxed compensator. Therefore, this MF-LQ problem is well-posed.


By Theorem \ref{th Hamil},
this MF-LQ
problem admits a unique solution satisfying the stochastic Hamiltonian system
\begin{equation}\label{Hsystem e2}
\left\{
    \begin{aligned}
    &0=-\beta u^*(t)+b(t) Y(t)+\tilde b(t)\mathbb E[Y(t)]+Z(t), \\
        &dX^*(t)=\big\{a(t)X^*(t)+\tilde a(t)\mathbb E[X^*(t)]+b(t)u(t)+\tilde b(t)\mathbb E[u(t)]\big\}dt+u(t)dW(t),\quad t\in [0,T], \\
        &dY(t)=-\big\{a(t)Y(t)+\tilde a(t)\mathbb E[Y(t)]+\alpha X^*(t)\big\}+Z(t)dW(t),\quad t\in [0,T], \\
        &X(0)=x,~~~Y(T)=\gamma X(T).
    \end{aligned}
\right.
\end{equation}

From the relationship \eqref{Sec2.F_Relat_YZ_Xu} in Proposition \ref{pro34}, we decouple equation  \eqref{Hsystem e2} as follows 
 \begin{equation}\label{example-yz}
\left\{
\begin{aligned}
& Y(t) = P(t) \big( X^*(t) -\mathbb E[X^*(t)] \big) +\widehat P(t) \mathbb E[X^*(t)],\\
&\mathbb E[Y(t)]=\widehat P(t) \mathbb E[X^*(t)],\\
& Z(t) = P (t)u^*(t) ,~~~\mathbb E[Z(t)]= P (t)\mathbb E[u^*(t)],\\
\end{aligned}
\right.
\end{equation}
where $(P, \widehat P)$  is the unique solution to the following pair of Riccati equations 
\begin{equation}\label{R P i}
    \left\{\begin{aligned}
            &~\dot P(t)+2a(t)P(t)+\alpha-\frac{\vert b(t)\vert^2 P(t)^2}{ P(t)-\beta}=0,\\
            &~P(T)= \gamma,~~~~P(t)-\beta>0,
        \end{aligned}\right.
\end{equation}
and
\begin{equation}\label{R P ii}
    \left\{\begin{aligned}
            &~\dot {\widehat P}(t)+  2\hat a(t) \widehat P(t)+\alpha-\frac{\vert \hat b(t)\vert^2 \widehat P(t)^2}{ P(t)-\beta}=0,\quad t\in [0,T],
            \\
            &~\widehat P(T)=\gamma.
        \end{aligned}\right.
\end{equation}
Putting \eqref{example-yz} into the first equation in \eqref{Hsystem e2} yields 
\[
\begin{aligned}
	-\beta u^*(t)+&b(t) [P(t) \big( X^*(t) -\mathbb E[X^*(t)] \big) +\widehat P(t) \mathbb E[X^*(t)]]+\tilde b(t)\widehat P(t) \mathbb E[X^*(t)]+P u^*(t)=0,
	\end{aligned}
\]
then the optimal control can be presented by 
\begin{equation}\label{u-x-pp}
\begin{aligned}
            u^*(t)&=-\frac{1}{P(t)-\beta}\big[b(t)P(t)\big(X^*(t)-\mathbb E[X^*(t)]\big)+\big(b(t)+\tilde b(t)\big)\widehat P(t)\mathbb E[X^*(t)]\big)\big],
            \end{aligned}
\end{equation}
where $X^*(\cdot)$ satisfies the following equation
\begin{equation}\label{u and x}
\left\{
    \begin{aligned}
       & dX^*(t)=\Big\{\big[a(t)-\frac{\vert b(t)\vert^2P(t)}{P(t)-\beta}\big]X^*(t)+\big[\tilde a(t)-\frac{1}{P(t)-\beta}\big(\vert b(t)+\tilde b(t)\vert^2\widehat P(t)\\
        &\quad-\vert b(t)\vert^2P(t)\big)\big]\mathbb E[X(t)]\Big\}dt-\frac{1}{P(t)-\beta}\big\{b(t)P(t)\big(X^*(t)-\mathbb E[X^*(t)]\big)\\
        &\quad+\big(b(t)+\tilde b(t)\big)\widehat P(t)\mathbb E[X^*(t)]\big)\big\}dW(t),\quad t\in [0,T], \\
        &X^*(0)=x.
    \end{aligned}
\right.
\end{equation}
We can see that the optimal control $u^*$ is determined by the system states $X^*(\cdot)$, $\mathbb E[X^*(\cdot)]$ and the solutions $P(\cdot)$, $\widehat P(\cdot)$ of Riccati equations.

 Moreover, 
combining  \eqref{example-yz} with \eqref{u-x-pp}, the unique solution $(Y,\mathbb E[Y],Z,\mathbb E[Z])$ of MF-FBSDE \eqref{e2 fbsde1}  can be represented as
 \begin{equation}\label{example-yz-Pu}
\left\{
\begin{aligned}
& Y(t) = P(t) \big( X^*(t) -\mathbb E[X^*(t)] \big) +\widehat P(t) \mathbb E[X^*(t)],\\
&\mathbb E[Y(t)]=\widehat P(t) \mathbb E[X^*(t)],\\
& Z(t) = -\frac{P(t)}{P(t)-\beta}\big[b(t)P(t)\big(X^*(t)-\mathbb E[X^*(t)]\big)+\big(b(t)+\tilde b(t)\big)\widehat P(t)\mathbb E[X^*(t)]\big)\big],\\
        &\mathbb E[Z(t)] = -\frac{P(t)}{P(t)-\beta}\big(b(t)+\tilde b(t)\big)\widehat P(t)\mathbb E[X^*(t)],\\
\end{aligned}
\right. 
\end{equation}
which also can be expressed by 
$P(\cdot)$, $\widehat P(\cdot)$, $X^*(\cdot)$ and $\mathbb E[X^*(\cdot)]$. In fact, \eqref{u and x} and \eqref{example-yz-Pu} provide an effective way for solving MF-FBSDE \eqref{Hsystem e2}.

In addition, we would like to discuss more about Hamiltonian system \eqref{Hsystem e2} with cases $\beta > 0$ and $\beta=0$.

 {\bf Case I:} When $\beta > 0$, Hamiltonian system \eqref{Hsystem e2} is rewritten as 
\begin{equation}\label{e2 fbsde1}
\left\{\begin{aligned}
        dX^*(t)&=\bigg\{a(t)X^*(t)+\tilde a(t)\mathbb E[X^*(t)]+\frac{b(t)}{\beta}\big[b(t)Y(t)+\tilde b(t)\mathbb E[Y(t)]+Z(t)\big]\\
       &~~~+\frac{\tilde b(t)}{\beta}\big[(b(t)+\tilde b(t))\mathbb E[Y(t)]+\mathbb E[Z(t)]\big]\bigg\}dt+\frac{1}{\beta}\big\{b(t)Y(t)+\tilde b(t)\mathbb E[Y(t)]\\
        &~~~+Z(t)\big\}dW(t), \quad t\in [0,T], \\
        dY(t)&=-\big\{a(t)Y(t)+\tilde a(t)\mathbb E[Y(t)]+\alpha X^*(t)\big\}dt\\
        &~~~+Z(t)dW(t), \quad t\in [0,T], \\
        X(0)&=x,~~~Y(T)=\gamma X(T).
    \end{aligned}\right.
\end{equation}
 
It is obvious that MF-FBSDE \eqref{e2 fbsde1} does not
satisfy the  monotonicity condition in \cite{Bensoussan-Yam-Zhang-2015}. Based on the above discussion, it follows from Theorem
\ref{th Hamil} that equation \eqref{e2 fbsde1} admits a unique solution. Moreover, the optimal control 
 \begin{equation}\label{example-u}
	u^*(t)=\frac{1}{\beta}\big[b(t) Y(t)+\tilde b(t)\mathbb E[Y(t)]+Z(t)\big]
\end{equation}
 can be expressed by $(Y(\cdot),\mathbb E[Y(\cdot)],Z(\cdot))$  in terms of \eqref{example-yz-Pu}. In fact, \eqref{example-u} is equivalent to \eqref{u-x-pp}. 
 
 {\bf Case II:}  When $\beta=0$, the Hamiltonian system \eqref{Hsystem e2} can be reduced to the following MF-FBSDE
\begin{equation}\label{e2 fbsde2}
\left\{
    \begin{aligned}
        dX(t)&=\{a(t)X(t)+\tilde a(t)\mathbb E[X(t)]+b(t)u(t)+\tilde b(t)\mathbb E[u(t)]\}dt+u(t)dW(t),\quad t\in [0,T],\\
        dY(t)&=-\big\{a(t)Y(t)+\tilde a(t)\mathbb E[Y(t)]+\alpha X(t)\big\}dt-\big\{b(t)Y(t)+\tilde b(t)\mathbb E[Y(t)]\big\}dW(t),~~ t\in [0,T],\\
        X(0)&=x,~~~Y(T)=\gamma X(T).
    \end{aligned}
\right.
\end{equation}
In \eqref{e2 fbsde2}, there are three unknown processes
$X(\cdot),Y(\cdot),\\u(\cdot)$, and the diffusion of the backward equation
depending on $Y(\cdot)$ and $\mathbb E [Y(\cdot)]$ while not  $Z(\cdot)$. This implies that \eqref{e2 fbsde2} is not
a classic FBSDE. To the best of our knowledge, this kind of equations are largely underexplored.
In this paper, because of the presence of
relaxed compensator, from Theorem \ref{th Hamil}, MF-FBSDE \eqref{e2
fbsde2} admits a unique solution. Moreover, the state solution $X^*(\cdot)$ is presented in \eqref{u and x}, $Y(\cdot)$ is solved by \eqref{example-yz-Pu}, and the optimal feedback $u^*(\cdot)$ is in the form of \eqref{u-x-pp} with $\beta=0$. 

For illustrating intuitively, we give simulations of numerical solutions by the Fig. 1. Taking $T=1$, $x=1$, $a=0.8$, $\tilde a=0.6$, $b=0.4$, $\tilde b=0.1$, $\alpha=0.5$, $\beta=0.2$ and $\gamma=1$ 
. Fig.1 (a) shows the numerical solutions $P$ and $\widehat P$ of Riccati equations \eqref{R P i}-\eqref{R P ii}, which are solved by Euler's method; Fig.1 (b) shows the optimal state $X^*$ and mean-value $\mathbb E[X^*]$; Fig.1 (c) presents the optimal control $u^*$ determined by $P$, $\widehat P$, $X^*$ and $\mathbb E[X^*]$. Moreover, $Y$, $\mathbb E[Y]$, $Z$ and $\mathbb E[Z]$ in Fig.1 (d)-(e) are described by $P$, $\widehat P$, $X^*$ and $\mathbb E[X^*]$ in Fig.1 (a)-(b). 

\begin{figure}[htbp]
\centering

\subfigure[]{
\begin{minipage}[t]{0.35\linewidth}
\centering
\includegraphics[width=2.2in]{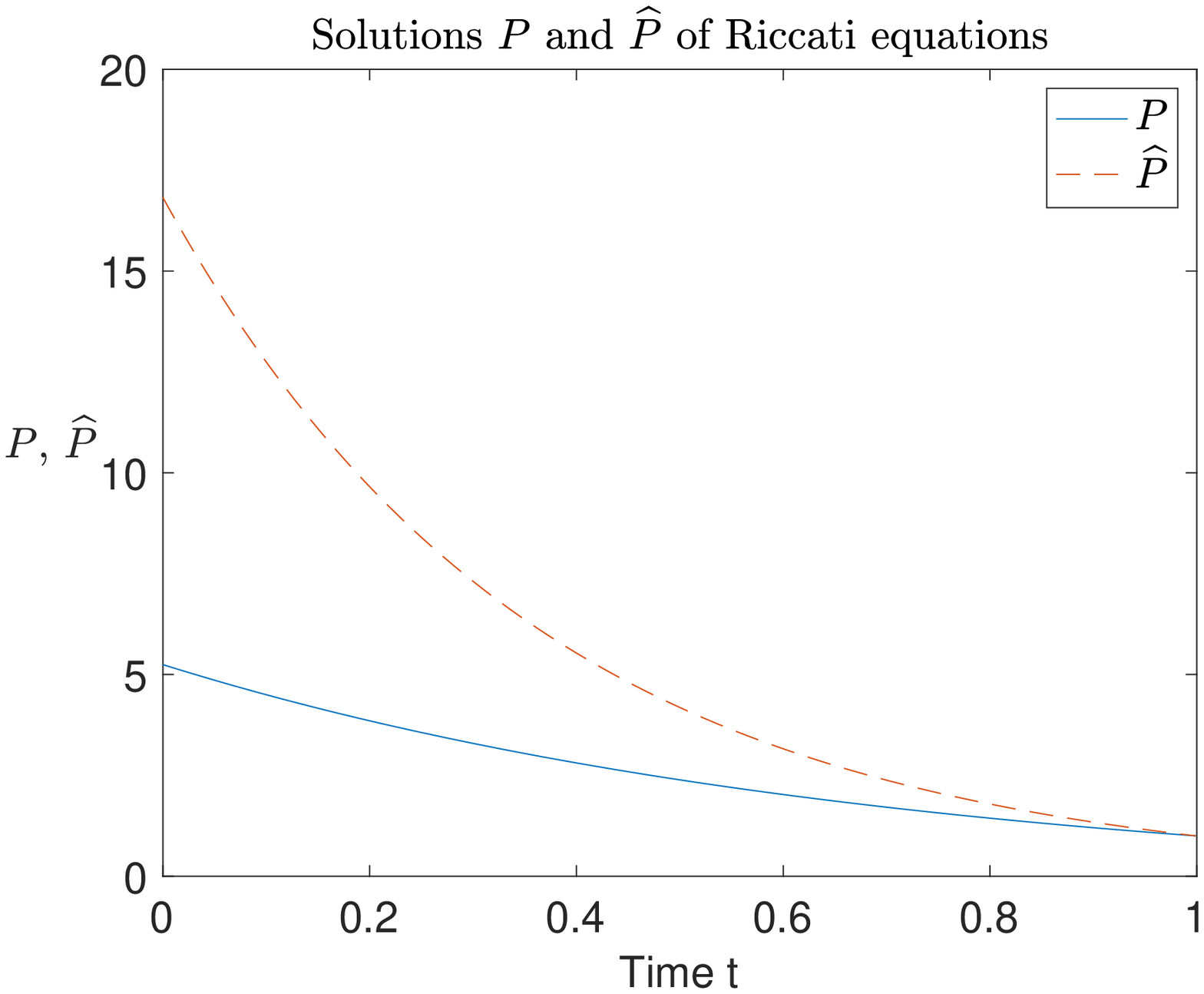}
\end{minipage}%
}%
\subfigure[]{
\begin{minipage}[t]{0.35\linewidth}
\centering
\includegraphics[width=2.2in]{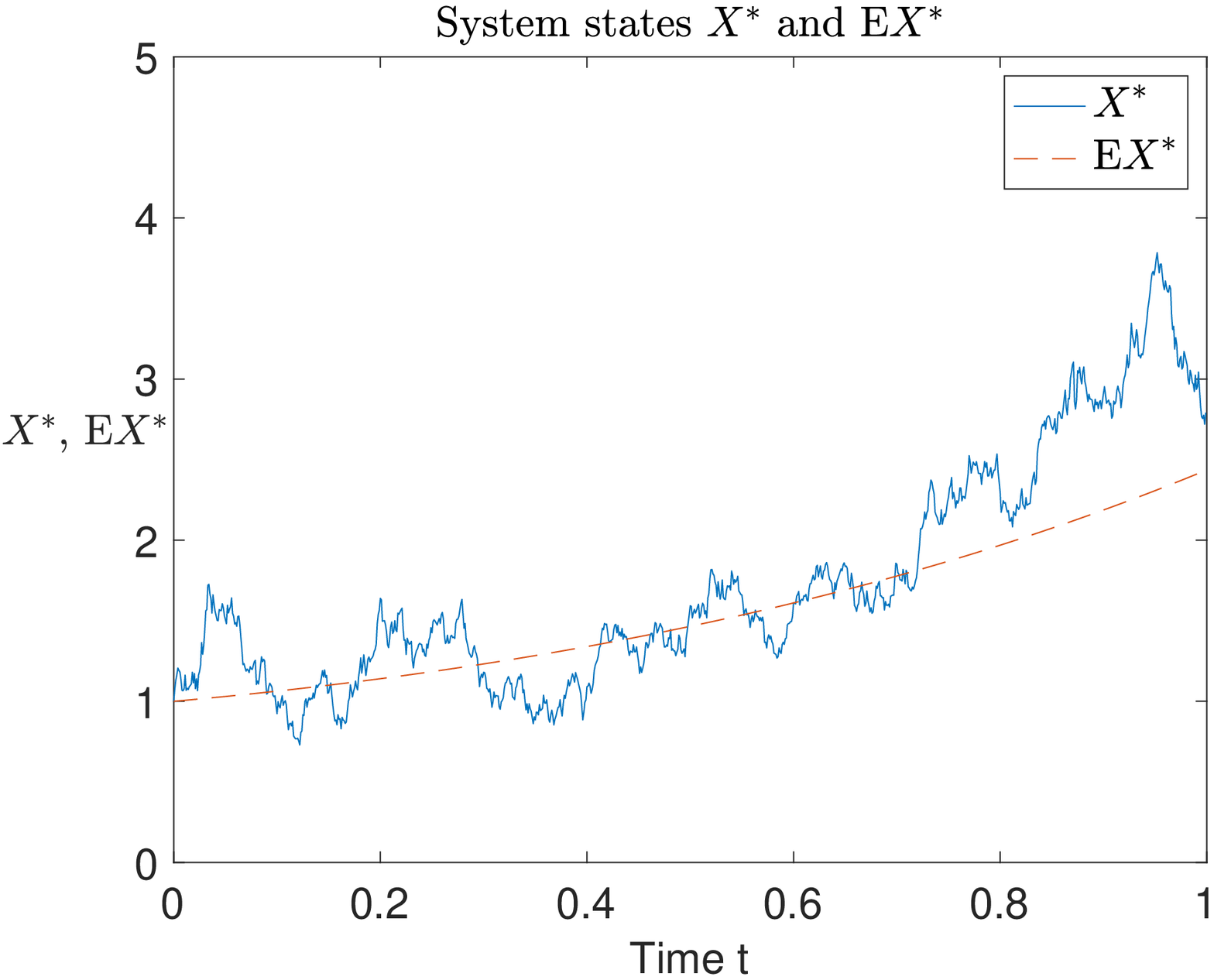}
\end{minipage}%
}%

\subfigure[]{
\begin{minipage}[t]{0.35\linewidth}
\centering
\includegraphics[width=2.2in]{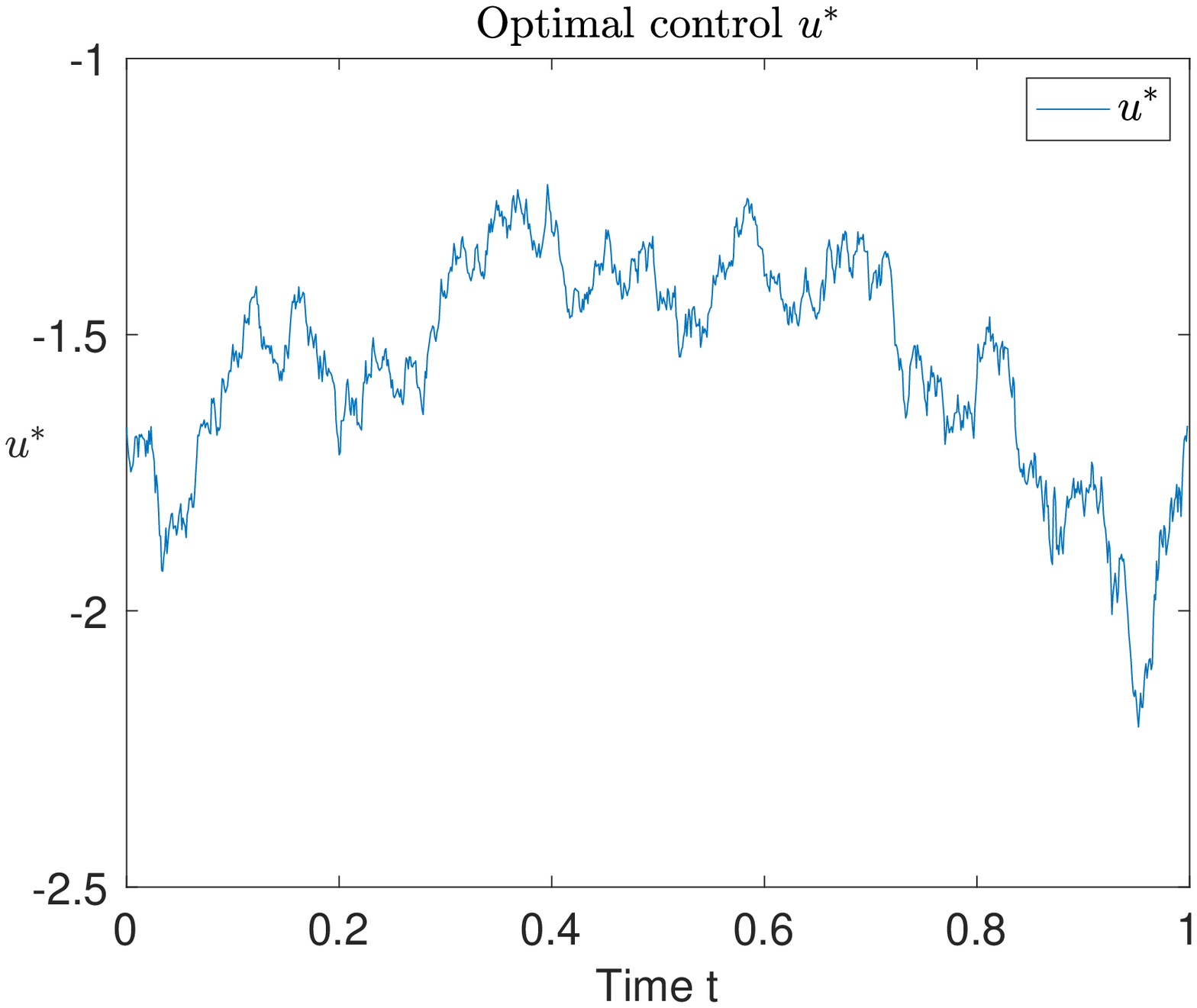}
\end{minipage}
}%
\subfigure[]{
\begin{minipage}[t]{0.35\linewidth}
\centering
\includegraphics[width=2.2in]{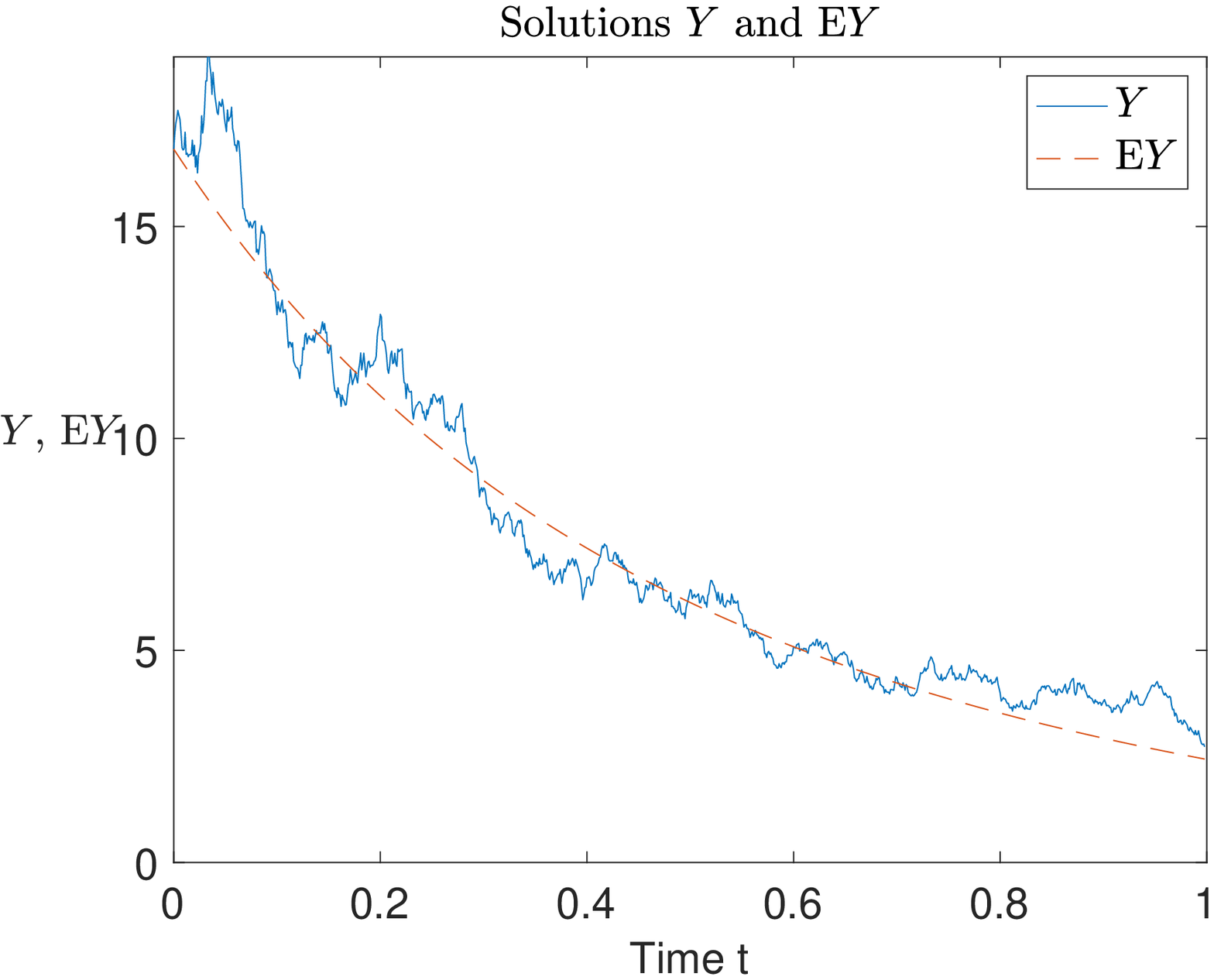}
\end{minipage}
}%
\subfigure[]{
\begin{minipage}[t]{0.35\linewidth}
\centering
\includegraphics[width=2.2in]{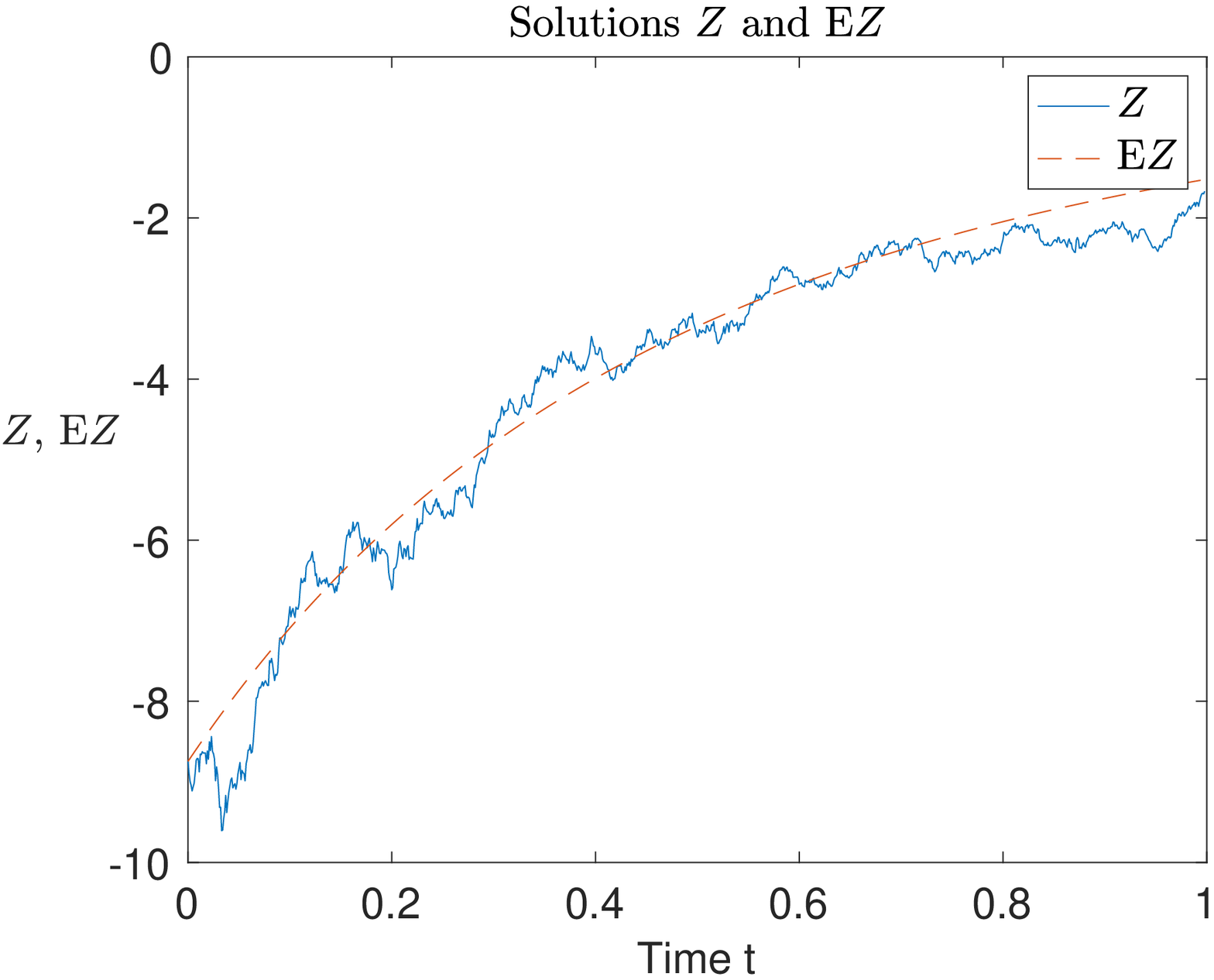}
\end{minipage}
}%
\centering
\caption{ Simulation results for solutions. (a): Solutions  $P$ and $\widetilde P$ of Riccati equations; (b):  Optimal system states $X^*$ and $\mathbb EX^*$; (c): Optimal control $u^*$; (d): Solutions $Y$, $\mathbb EY$; (e): Solutions $Z$ and $\mathbb EZ$.}
\end{figure}

 Let $\alpha\geq0$, $\gamma>\max\{\beta,0\}$ and $a(t) \geq (b(t)^2 \gamma)/(2(\gamma -\beta))$ as before. From another viewpoint, when $\gamma > \beta$, one can check the cost functional is uniformly convex in the control variable, which leads to the unique existence of optimal control. However, the uniform convexity is broken when $\gamma \leq \beta$. One can refer to Sun \cite{Sun-2017} for more details on this viewpoint.

\subsection{An example with negative definite cost weighting of control}\label{example3}

In this subsection, we study Example 1.2 presented in
Section 1. Firstly, we obtain the optimal controls in open-loop
form and closed-loop, respectively. Secondly, the
explicit solutions of MF-FBSDE and Riccati equations
are presented. Consider the following system
\begin{equation*}
\left\{
    \begin{aligned}
        dX(t)&=\big\{\alpha(t)X(t)+\widetilde \alpha(t)\mathbb E[X(t)]\big\}dt+\beta(t)u(t)dW(t), \\
        X(0)&=x,
    \end{aligned}
    \right.
\end{equation*}
and the cost functional
\begin{equation*}
    \begin{aligned}
        J(x;u(\cdot))=\mathbb E\int_0^T\big\{\gamma(t)X^2(t)+\widetilde\gamma(t)(\mathbb E[X(t)])^2-\theta(t) u^2(t)\big\}dt+G\mathbb E[X^2(T)].
    \end{aligned}
\end{equation*}
Here, assume that all the coefficients are deterministic.
Moreover, $\alpha(\cdot)$, $\widetilde\alpha(\cdot)$,
$\beta(\cdot)$, $\widetilde\beta(\cdot)$, $\gamma(\cdot)$,
$\widetilde\gamma(\cdot)$ are non-negative. In particular, $\theta(\cdot)$ is
positive but not be too large, and satisfies
\[
\theta(t)<G\beta^2(t) e^{\int_0^T2\alpha(s)ds}-\int_0^T\gamma(s)e^{\int_t^s2\alpha(\tau)d\tau}ds,
\]
and $G$ is non-negative.
%
The corresponding Riccati equations follow
\begin{equation*}\label{P i1}
\left\{\begin{aligned}
            &~\dot P(t)+2\alpha(t)P(t)+\gamma(t)=0,\\
            &~P(T)= G,\\
            &~\beta^2(t)P(t)-\theta(t)>0,
\end{aligned}\right.
\end{equation*}
and
\begin{equation*}\label{hatP ii1}
\left\{\begin{aligned}
            &~\dot {\widehat P}(t)+2[\alpha(t)+\widetilde \alpha(t)]\widehat P(t)+\gamma(t)+\widetilde\gamma(t)=0,
            \\
            &~\widehat P(T)= G.
\end{aligned}\right.
\end{equation*}
A short calculation yields
\begin{equation*}
    P(t)=G e^{\int_0^T2\alpha(s)ds}-\int_0^T\gamma(s)e^{\int_t^s2\alpha(\tau)d\tau}ds,
\end{equation*}
and
\begin{equation*}
    \widehat P(t)=G e^{\int_0^T2(\alpha(s)+\widetilde\alpha(s))ds}-\int_0^T(\gamma(s)+\widetilde\gamma(s))e^{\int_t^s2(\alpha(\tau)+\widetilde\alpha(\tau))d\tau}ds.
\end{equation*}
We choose $(P(\cdot),\widehat P(\cdot))$ as the relaxed compensator. This problem is well-posed.

From Theorem \ref{Th Rel}, the closed-loop optimal control is taken by $u^*(t)=0$. Also,
from Theorem \ref{th Hamil}, the open-loop optimal control can be presented by
\begin{equation}\label{ue3open}
    u^*(t)=\frac{\beta(t)}{\theta(t)}Z(t),
\end{equation}
where $Z(\cdot)$ is determined by
\begin{equation}\label{FBSDEe3}
\left\{\begin{aligned}
        dX(t)&=\big\{\alpha(t)X(t)+\widetilde \alpha(t)\mathbb E[X(t)]\big\}dt+\frac{\beta^2(t)}{\theta(t)}Z(t)dW(t),\\
        dY(t)&=-\big\{\alpha(t)Y(t)+\widetilde\alpha(t)\mathbb E[Y(t)]+\gamma(t)X(t)+\widetilde\gamma(t)\mathbb E[X(t)]\big\}dt+Z(t)dW(t),\\
        X(0)&=x,~~~Y(T)=G X(T).
\end{aligned}\right.
\end{equation}
Comparing two forms of optimal control, we get $Z(\cdot)=0$.

Next, we solve $\mathbb E[X(t)] $ and $X(t)$ from \eqref{FBSDEe3}, there are
\begin{equation*}
    \mathbb E [X(t)]=xe^{\int_0^t\big(\alpha(s)+\widetilde\alpha(s)\big)ds}
\end{equation*}
and
\begin{equation}\label{xe3}
     X(t)=xe^{\int_0^t\alpha(s)ds}\big(1+\int_0^t\widetilde\alpha(s)e^{\int_0^s\widetilde\alpha(\tau)d\tau}ds\big).
\end{equation}
It follows from \eqref{Sec2.F_Relat_YZ_Xu} in Proposition \ref{pro34} that
\begin{equation}\label{ye3}
\begin{aligned}
    Y(t)&=\big(G e^{\int_0^T2\alpha(s)ds}-\int_0^T\gamma(s)e^{\int_t^s2\alpha(\tau)d\tau}ds\big)(X(t)-\mathbb E[X(t)])\\
    &~~~+\big(G e^{\int_0^T2(\alpha(s)+\widetilde\alpha(s))ds}-\int_0^T(\gamma(s)+\widetilde\gamma(s))e^{\int_t^s2(\alpha(\tau)+\widetilde\alpha(\tau))d\tau}ds\big)\mathbb E[X(t)]\\
    &=xe^{\int_0^t\alpha(s)ds}\big(G e^{\int_0^T2\alpha(s)ds}-\int_0^T\gamma(s)e^{\int_t^s2\alpha(\tau)d\tau}ds\big)\big(1+\int_0^t\widetilde\alpha(s)e^{\int_0^s\widetilde\alpha(\tau)d\tau}ds-e^{\int_0^t\widetilde\alpha(s)ds}\big)\\
    &~~~+xe^{\int_0^t\big(\alpha(s)\widetilde\alpha(s)\big)ds}+\big(G e^{\int_0^T2(\alpha(s)+\widetilde\alpha(s))ds}-\int_0^T(\gamma(s)+\widetilde\gamma(s))e^{\int_t^s2(\alpha(\tau)+\widetilde\alpha(\tau))d\tau}ds\big).
    \end{aligned}
\end{equation}
Now, it follows from \eqref{xe3}, \eqref{ye3} and $Z(\cdot)=0$ that $(X(\cdot),Y(\cdot),Z(\cdot))$ is the solution to \eqref{FBSDEe3}.

\end{document}